\documentclass[10pt]{article} % default is 10 pt

\usepackage{latexsym,amsmath,amssymb,a4wide,amscd,authblk}
\usepackage[latin1]{inputenc}
\usepackage[T1]{fontenc}

\usepackage{array}

\usepackage{enumerate}

\usepackage{amsthm}

\usepackage{lscape}

\usepackage{url}

\usepackage{mathabx}

\usepackage{longtable}
\usepackage{multirow}

\usepackage{MnSymbol}

\usepackage{amsthm}
\numberwithin{equation}{section}
\numberwithin{figure}{section}
\theoremstyle{plain}
\newtheorem {theorem}{Theorem}[section]

\newtheorem {lemma}[theorem]{Lemma}

\newtheorem {kor}[theorem]{Corollary}

\theoremstyle{definition}
\newtheorem {definition}[theorem]{Definition}
\theoremstyle{remark}
\newtheorem {remark}[theorem]{Remark}

\usepackage{accents}
\newcommand\thickbar[1]{\accentset{\rule{.7em}{.8pt}}{#1}}

\newcommand{\SB}{\thickbar{S}}
\newcommand{\TB}{\thickbar{T}}

\newcommand{\link}{\operatorname{lk}}
\newcommand{\str}{\operatorname{st}}

\usepackage[pdftex]{graphicx} % needed for including graphics e.g. EPS, PS

\usepackage{pdfpages}

\usepackage[nooneline]{caption}

\begin{document}

\date{}

\renewcommand{\Authfont}{\scshape\small}
\renewcommand{\Affilfont}{\itshape\small}
\renewcommand{\Authand}{,}
\renewcommand{\Authands}{, }

\title{\Large {\bf The Pachner graph of $2$-spheres}}

\author[1] {Benjamin A.~Burton}
\author[2] {Basudeb Datta}
\author[1] {Jonathan Spreer}

\affil[1] {School of Mathematics and Physics, The University of Queensland,
Brisbane QLD 4072, Australia. \texttt{bab@maths.uq.edu.au}, \texttt{jonathan.spreer@fu-berlin.de}.}
\affil[2] {Department of Mathematics, Indian Institute of Science,
Bangalore 560\,012, India. \texttt{dattab@iisc.ac.in}.}

\maketitle

{\em
It is well-known that the Pachner graph of $n$-vertex triangulated $2$-spheres is connected, i.e., each pair of $n$-vertex triangulated $2$-spheres can be turned into each other by a sequence of edge flips for each $n\geq 4$. In this article, we study various induced subgraphs of this graph. In particular, we prove that the subgraph of $n$-vertex {\em flag} $2$-spheres distinct from the double cone is still connected. In contrast, we show that the subgraph of $n$-vertex {\em stacked} $2$-spheres has at least as many connected components as there are trees on $\lfloor\frac{n-5}{3}\rfloor$ nodes with maximum node-degree at most four.
}

\medskip



\section{Introduction}

The {\em Pachner graph} of triangulated $2$-spheres is the graph, whose nodes are triangulated  $2$-spheres (also known as planar triangulations), and two nodes are connected by an arc if and only if their corresponding triangulations can be transformed into each other by a single {\em bistellar move}, i.e., an edge flip, a stellar subdivision of a triangle or its inverse, see Figure~\ref{fig:pachner2}.

The Pachner graph of triangulated $2$-spheres is connected. More precisely, starting from an arbitrary node representing an $n$-vertex $2$-sphere, a path of length $O(n)$ can be found in the Pachner graph ending at the node representing the boundary of the tetrahedron. Conversely, it is not difficult to see that $\Omega(n)$ arcs are also necessary for the length of such a path.

The Pachner graph has a natural graded structure into induced subgraphs on the sets of nodes representing $n$-vertex triangulated $2$-spheres, $n$ fixed: The arcs within a level correspond to edge flips, the arcs corresponding to stellar subdivisions (and their inverses) connect different levels of the grading. It is well-known that each such level, sometimes called the {\em flip graph (of $n$-vertex triangulated $2$-spheres)}, is connected \cite{Wagner1936FourColourThm}. Moreover, its diameter is bounded from above by $5n - 23$ due to work by Cardinal, Hoffmann, Kusters, T{\'o}th and Wettstein \cite{Cardinal15ArcDiagramsFlips} and bounded from below by $7n/3 - 34$ due to work by Frati \cite{Frati17LowerBound}. These two results are the most recent additions to a series of papers aimed at reducing the gap between upper bounds and lower bounds for the diameter of the flip graph. One of the current open problems in this area is to find an upper bound and a lower bound which are apart by a factor of two (the optimum achievable by bounding the diameter as twice the distance of a particular pair of triangulations). See \cite{Bose11FlipGraphSurvey} for a survey on previous attempts to bound the diameter of the flip graph of the $2$-sphere.

In \cite{Sulanke09IsoFreeEnumeration}, Sulanke and Lutz show that there are exactly $59$ twelve-vertex triangulations of the orientable surface of genus six. Since they all must be neighbourly, none of them allows any edge flips. Thus, the Pachner graph of twelve-vertex triangulated orientable surfaces of genus six is the discrete graph on $59$ nodes.

See various chapters of \cite{DeLoera10Triangulations} for further and closely related research concerning the flip graph and similar objects.

\medskip

Structural results for, as well as bounds on flip distances in Pachner graphs (of spheres or, more general, triangulated manifolds) which are as precise as the ones mentioned above, are unlikely to be provable in dimensions greater than two. For instance, the best upper bound for distances in the Pachner graph of {\em generalised triangulations of the $3$-sphere} is given by $O( t^2 2^{c t^2})$ for the number of moves between a $t$-tetrahedron triangulation of $S^3$ and the boundary of the $4$-simplex, see Mijatovi\'c \cite{Mijatovic03SimplTrigsOfS3}. Naturally, the corresponding upper bound in the simplicial setting must be at least as large. Moreover, the $n$-th level of the Pachner graph of simplicial triangulations of the $3$-sphere is not even connected (in contrast to the generalised setting, see \cite{Matveev07AlgorTopClassif3Mflds}): Consider an $n$-vertex triangulation of the $3$-sphere containing (i) no edge of degree three and (ii) the complete graph with $n$ vertices as edges. Such a triangulation only admits stellar subdivisions as bistellar moves and is thus isolated in the Pachner graph of $n$-vertex triangulated $3$-spheres. See \cite{Dougherty04Unflippable,Spreer14CyclicCombMflds} for a number of examples of such triangulated $3$-spheres.

Even more, in dimensions greater than three, no such general upper bounds can exist at all due to the undecidability of the homeomorphism problem.

\medskip

In this paper we focus on the connectedness of certain subgraphs of the Pachner graph of $n$-vertex triangulated $2$-spheres. Namely, we consider what are called {\em stacked} and {\em flag $2$-spheres} (see Sections~\ref{ssec:flag} and \ref{ssec:stacked} for details). In many ways, flag $2$-spheres are the counterpart to stacked $2$-spheres. While stacked $2$-spheres contain the maximum number of induced $3$-cycles, flag $2$-spheres do not contain any such cycle. Moreover, every triangulated $2$-sphere can be decomposed into a collection of flag $2$-spheres and boundaries of the tetrahedron (called {\em standard $2$-spheres}) by iteratively cutting along its induced $3$-cycles and pasting the missing triangles. For a flag $2$-sphere this decomposition is the $2$-sphere itself. For stacked $2$-spheres it yields the maximum number of connected components, each isomorphic to the standard $2$-sphere.

In \cite[Theorem~\ref{thm:mori}]{Mori03DiagonalFlips} the authors give upper bounds for the number of edge flips connecting two flag $2$-spheres within the class of
Hamiltonian triangulations. Our main result states that such a sequence of edge flips exists even {\em within the class of flag $2$-spheres} -- as long as both triangulations are distinct from the double cone $\Gamma_n$ over the $(n-2)$-gon (Figure~\ref{fig:An}(a)), see Theorem~\ref{thm:pachnerGraph}. Observe that excluding the $n$-vertex double cone $\Gamma_n$, $n\geq 6$, from Theorem~\ref{thm:pachnerGraph} is necessary: $\Gamma_n$ is a flag $2$-sphere in which every edge contains a degree four vertex. Thus every edge flip on $\Gamma_n$ produces a vertex of degree three and the resulting complex is not flag. In particular, $\Gamma_n$ cannot be connected to any other flag $2$-sphere by an edge flip.

This theorem complements a result by Lutz and Nevo stating that every pair of $d$-dimensional flag complexes, $d \geq 3$, is connected by a sequence of edge subdivisions, and edge contractions~\cite{Lutz16FlagComplexes}.

In contrast, the subgraph of the Pachner graph of $n$-vertex stacked $2$-spheres has much less uniform properties. In Section~\ref{sec:stacked} we give a precise condition on when exactly an edge flip of a stacked $2$-sphere produces another stacked $2$-sphere (Theorem~\ref{theo:bsm-s2s}). Using this result, we prove that the Pachner graph of $n$-vertex stacked $2$-spheres is not connected, and that there are at least as many connected components as there are trees on $\lfloor\frac{n-5}{3}\rfloor$ nodes and with degrees of nodes at most four. In particular, the number of connected components of the Pachner graph of $n$-vertex stacked $2$-spheres is exponential in $n$ (Corollary~\ref{coro:cor4}). Furthermore, we show that a pair of $n$-vertex stacked $2$-spheres can be connected by a sequence of $n$-vertex stacked $2$-spheres, each related to the previous one by an edge flip, if their associated stacked $3$-balls have a dual graph without degree four vertices (Theorem~\ref{thm:stackedcomp}). These results are complemented by additional experimental data for $n \leq 14$ vertices (Table~\ref{fig:stackedPachner}).

\medskip

Altogether, the results contained in this paper together with existing results on the flip graph discussed above allow us to draw a relatively precise map of the flip graph of $n$-vertex triangulated $2$-spheres. Having more knowledge about the structural properties of the flip graph might be one key for challenging future endeavours such as sampling triangulated $2$-spheres or even generating triangulated $2$-spheres with certain properties under some conditions of randomness.

For a graphical summary of what is known about the flip graph at present see Figure~\ref{fig:map}.

\begin{landscape}

\begin{figure}[htb]
 \centerline{\includegraphics[width=1.5\textwidth]{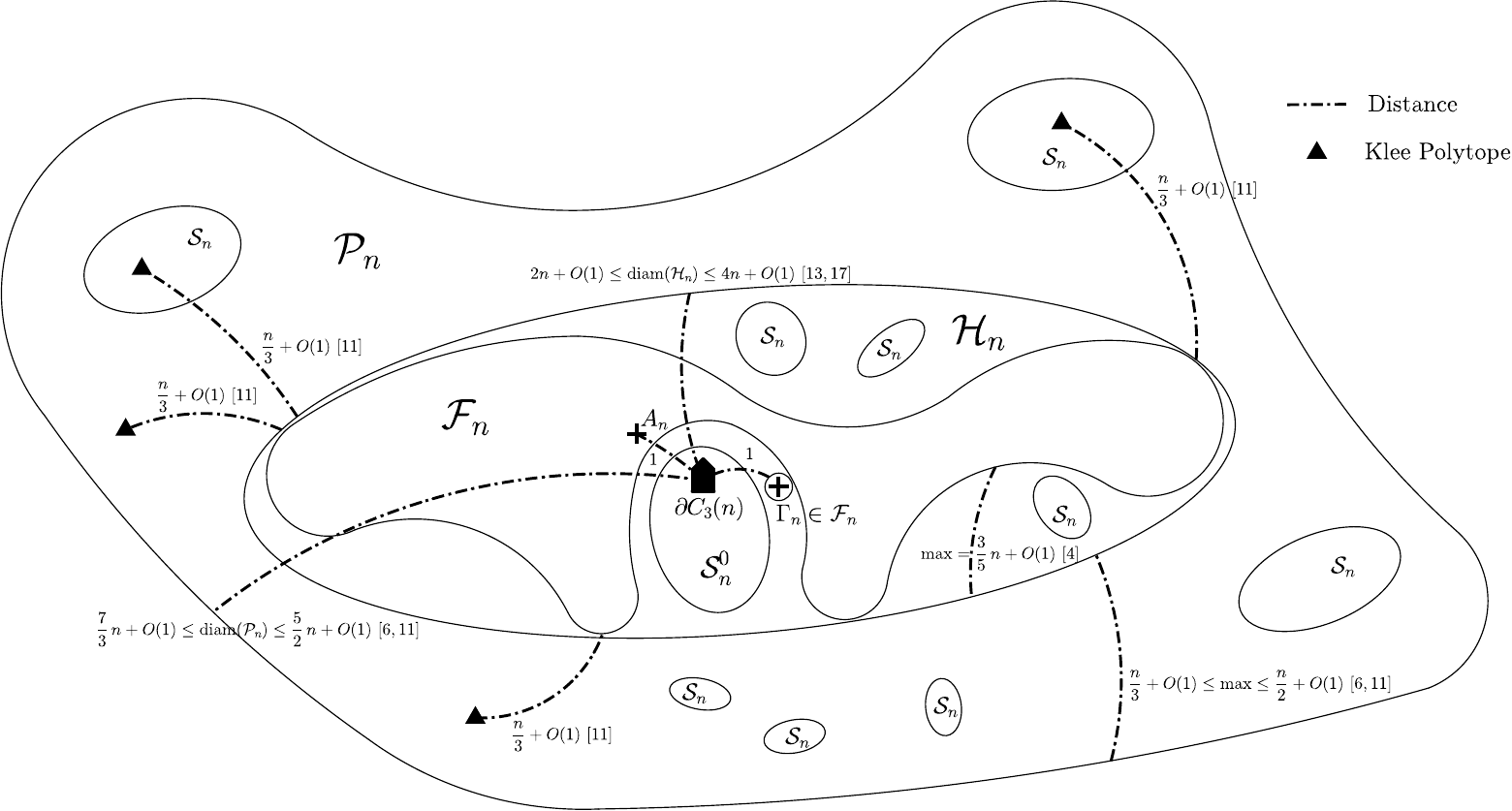}}
  \caption{Map of the flip graph $\mathcal{P}_n$ of $n$-vertex $2$-spheres. For notations see Sections~\ref{sec:prelims}, \ref{sec:flag} and \ref{sec:stacked}. The containment of $\mathcal{S}_n^0 \subset \mathcal{H}_n$ is due to a simple but unpublished argument which is left as an exercise to the interested reader. A {\em Klee polytope} is a triangulated $2$-sphere $S'$ obtained from a triangulated $2$-sphere $S$ by stellarly subdividing each triangle of $S$ once.  \label{fig:map}}
\end{figure}

\end{landscape}

\subsection*{Acknowledgement}

This work was supported by DIICCSRTE, Australia (project AISRF06660) and DST, India (DST/INT/AUS/P-56/2013(G)), under the Australia-India Strategic Research Fund.
The second author was also supported by SERB, DST (Grant No. MTR/2017/ 000410).
The third author thanks DIICCSRTE for a postdoctoral fellowship. In addition, for parts of this project the third author was supported by the Einstein Foundation (project ``Einstein Visiting Fellow Santos'').

\section{Preliminaries} \label{sec:prelims}

\subsection{Triangulations of $2$-spheres}
\label{ssec:trigs}

A {\em triangulation of the $2$-sphere}, sometimes also referred to as a {\em planar triangulation}, is an $n$-vertex graph embedded in the $2$-sphere with $3n-6$ edges for some $n\geq 4$. As a direct result, the embedding decomposes the $2$-sphere into $2n-4$ triangles. This graph together with the triangles is called a {\em triangulated $2$-sphere}. The graph is also called the {\em edge graph} of the triangulated $2$-sphere.
The simplest example of a triangulated $2$-sphere is the boundary of the tetrahedron, called the {\em standard $2$-sphere}.

Every $n$-vertex triangulated $2$-sphere can be identified with an {\em abstract simplicial complex}, that is, a set of subsets of a finite ground set $V$, called {\em faces}, closed under taking subsets. For this, label its vertices with the elements of $V=\{ 1, \ldots,  n \}$ and represent triangles, edges and vertices by subsets of $V$ of cardinality three, two and one respectively. Note that, for the purpose of this article, we sometimes do not make the distinction between vertices of an abstract simplicial complex and elements of its ground set.

We say that two triangulated $2$-spheres are {\em combinatorially isomorphic}, or just {\em isomorphic} for short, if their respective abstract simplicial complexes are equal possibly after relabeling the elements of the ground set. In this article, whenever we talk about triangulated $2$-spheres we mean their corresponding isomorphism classes of abstract simplicial complexes.
By a theorem of Steinitz \cite{Steinitz22Realisation}, isomorphism types of triangulated $2$-spheres are in one-to-one correspondence with isomorphism types of simplicial $3$-polytopes. A fact which does not generalise to higher dimensions \cite{Barnette73NonPolytopalSphere,Gruenbaum67Enum8vtxSpheres}.

Given a triangulated $2$-sphere $S$, we usually denote its set of vertices, edges and triangles by $V(S)$, $E(S)$ and $F(S)$ respectively. Analogous notation is used for arbitrary abstract simplicial complexes. For $v \in V(S)$, its {\em star} $\str_S (v)$ is the simplicial complex generated by all triangles in $F(S)$ containing $v$. The edges and vertices of $\str_S (v)$ not containing $v$ (i.e., the boundary of $\str_S (v)$) constitute the {\em link} of $v$ in $S$, written $\link_S (v)$. The star and the link of an arbitrary face of an arbitrary abstract simplicial complex are defined analogously. The number of edges containing $v$ is called the {\em degree} of $v$, written $\deg_S (v)$.

For a triangulated $2$-sphere $S$ on ground set $V$ and $W \subseteq V$, the {\em subcomplex induced by $W$}, denoted $S[W]$, is the simplicial complex of all triangles, edges and vertices of $S$ entirely contained in $W$. Induced subcomplexes on arbitrary abstract simplicial complexes are defined analogously. In the special case of a graph $G=(V,E)$ and one of its vertices $v \in V$, the induced subgraph $G[V\setminus\{v\}]$ is referred to as the {\em vertex-deleted subgraph $G-v$}.

\subsection{Flag and Hamiltonian $2$-spheres}
\label{ssec:flag}

There are several special types of triangulated $2$-spheres which are relevant for this article. The most important ones are introduced in this section and in Section~\ref{ssec:stacked}.

\begin{definition}[Flag $2$-sphere]
	A {\em flag $2$-sphere} is a triangulated $2$-sphere in which all minimal non-faces of the underlying simplicial complex are of size two. Equivalently, a flag $2$-sphere is a triangulated $2$-sphere distinct from the standard $2$-sphere, in which every $3$-cycle (i.e., cycle of three edges) bounds a triangle.
\end{definition}

Every triangulated $2$-sphere $S$ can be decomposed into a collection of flag $2$-spheres and standard $2$-spheres: Simply cut along a $3$-cycle not bounding a triangle, and fill in the missing triangle in both parts. Iterating this procedure results in a set of spheres called the primitive components of $S$. Identifying each one of them by a node, and the $3$-cycles by arcs between nodes this defines a tree. If the tree is a single vertex, $S$ is called {\em primitive}. A triangulated $2$-sphere is called {\em $4$-connected} if its edge graph is $4$-connected. A triangulated $2$-sphere distinct from the standard $2$-sphere is $4$-connected if and only if it is primitive if and only if it is flag.

\begin{definition}[Hamiltonian $2$-sphere]
	A {\em Hamiltonian $2$-sphere} is a triangulated $2$-sphere containing a Hamiltonian cycle in its edge graph.
\end{definition}

Hamiltonian $2$-spheres play an important role in the proofs of upper bounds for the diameter of the Pachner graph of $n$-vertex triangulated $2$-spheres for a fixed $n$, see \cite{Bose11FlipGraphSurvey} for an overview. This is due to (i) the well-behaved structure of the Pachner graph of $n$-vertex Hamiltonian $2$-spheres which admits relatively precise bounds on its diameter, see Theorem~\ref{thm:mori}, and (ii) the fact that a flag $2$-sphere is necessarily Hamiltonian \cite{Whitney31Planar4ConnGraphsAreHamiltonian}. The converse of (ii) is not true.

\subsection{Stacked $3$-balls and stacked $2$-spheres}
\label{ssec:stacked}

A {\em triangulated $3$-ball} is a collection of tetrahedra (together with their faces) whose union is a topological $3$-ball. If $B$ is a triangulated $3$-ball then its {\em boundary} $\partial B$ is the complex generated by all triangles of $B$ contained in only one tetrahedron of $B$. By the {\em standard $3$-ball} we mean a single tetrahedron together with its faces. The boundary of the standard $3$-ball is the standard $2$-sphere.

A triangulated $3$-ball $B$ is called a {\em stacked $3$-ball} if there is a sequence $B_1, \dots, B_m$ of triangulated $3$-balls such that $B_1$ is the standard $3$-ball, $B_m=B$ and, for $2\leq i\leq m$, $B_i$ is constructed from $B_{i-1}$ by gluing (or stacking) a standard $3$-ball onto a single triangle of $B_{i-1}$. Note that, by construction, all edges and vertices of $B$ are contained in $\partial B$. 

Conversely, let $B$ be a triangulated  $3$-ball with all of its edges and vertices in $\partial B$. If $t$ is an interior triangle in $B$ then the boundary of $t$ is a $3$-cycle in $\partial B$ (i.e., an induced $3$-cycle in $\partial B$). Since $B$ is a union of tetrahedra, by Lemma \ref{coro:interiorface} below, $B$ is the union of two smaller $3$-balls $B_1$ and $B_2$ glued together along $t$
and all the edges and vertices of $B_i$ are in $\partial B_i$ for $i=1, 2$. Inductively, this shows that $B$ is a stacked $3$-ball. (See \cite[Theorem 4.5]{DattaMurai14StackedTriangulations} for a more general result with a rigorous proof.)
A {\em stacked $2$-sphere} is a triangulated $2$-sphere isomorphic to the boundary of a stacked $3$-ball. It follows from the definition of a stacked ball that an $n$-vertex stacked $2$-sphere contains exactly $n-4$ induced $3$-cycles.

For an abstract simplicial complex $C$ whose faces consist of tetrahedra and their subfaces, the graph whose nodes correspond to the tetrahedra of $C$ and two nodes are connected by an arc if and only if their corresponding tetrahedra share a triangle is called the {\em dual graph} of $C$, denoted by $\Lambda(C)$. If $B$ is a stacked $3$-ball then $\Lambda(B)$ is a tree, and every node of $\Lambda (B)$ corresponds to a primitive component of the bounding stacked $2$-sphere $\partial B$. It follows that a triangulated $2$-sphere is stacked if and only if all of its primitive components are standard $2$-spheres.

%From \cite[Lemma 2.1]{Datta12InfFamTightTrig} we have the following additional characterisation of stacked $3$-balls.
%
%\begin{lemma}[Datta, Singh \cite{Datta12InfFamTightTrig}] \label{prop:ds-stackedball}
%Let $C$ be an $n$-vertex simplicial complex whose inclusion-maximal faces are tetrahedra. Then $C$ is a stacked $3$-ball if and only if the dual graph $\Lambda(C)$ is a tree and the number of tetrahedra of $C$ equals $n-3$.
%\end{lemma}

The following lemma is a corollary of \cite[Lemma 3.4]{DattaMurai14StackedTriangulations} which is proved for arbitrary dimension and in the more general setting of homology balls.

\begin{lemma} \label{coro:interiorface}
Let $B$ be a stacked $3$-ball. If $t$ is an interior triangle of $B$ then the induced complex $B[V(B)\setminus t]$ has exactly two connected components. Moreover, if $u$ and $v$ are the two apices of tetrahedra of $B$ containing $t$, then $u$ and $v$ are in different components of $B[V(B)\setminus t]$.
\end{lemma}

From \cite[Lemma 4.6 and Remark 4.1]{Bagchi08LBTNormPseudoMnf} we know the following statement.

\begin{lemma}[Bagchi, Datta \cite{Bagchi08LBTNormPseudoMnf}] \label{prop:bd-lbt}
Let $S$ be a stacked $2$-sphere with edge graph $G$. Let $\SB$ denote the simplicial complex whose faces are all the cliques of $G$. Then $\SB$ is a stacked $3$-ball and $S = \partial \SB$. Moreover, up to isomorphism, $\SB$ is the unique stacked $3$-ball such that $S = \partial \SB$.
\end{lemma}

\subsection{Bistellar moves}
\label{ssec:bistellar}

Bistellar moves are local combinatorial alterations of a simplicial complex which, in general, change the isomorphism type of the complex, but not the topology of the underlying space. For a triangulated $2$-sphere $S$ there are the following two bistellar moves to consider (see also Figure~\ref{fig:pachner2}).

\begin{itemize}
  \item Replace a triangle of $S$ by three triangles joined around a new vertex. Such a stellar subdivision of a triangle is also called a {\em $0$-move} (because a $0$-dimensional face is inserted) or {\em $1\mbox{-}3$-move} (because one triangle is replaced by three triangles). For its inverse operation, a so-called {\em $2$-move} (a $2$-dimensional face is inserted) or  {\em $3\mbox{-}1$-move} (three triangles are replaced by one), remove the vertex star of a vertex of degree three and replace it by a single triangle. This inverse operation is only possible if the new triangle is not already present in the triangulation. In particular, the standard $2$-sphere does not allow any $2$-moves.

  \item Replace two triangles of $S$ which are joined along a common edge, say $abx$ and $aby$, and replace them with triangles $axy$, $bxy$. This operation is possible if and only if $xy$ is not an edge of $S$. This move is called a {\em $1$-move}, {\em $2\mbox{-}2$-move}, or, for obvious reasons, an {\em edge flip}. Throughout this article we denote it by $ab \mapsto xy$. The inverse of an edge flip is again an edge flip.
\end{itemize}

\begin{figure}[!h]
  \begin{center}
   \includegraphics[width=0.75\textwidth]{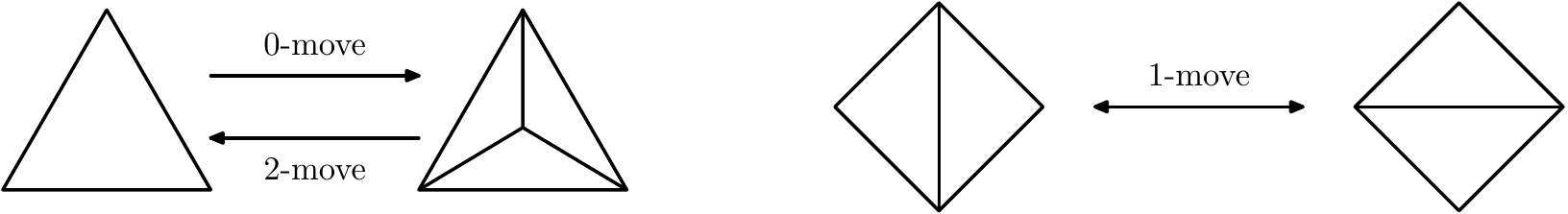}
    \caption{The bistellar moves in dimension two. \label{fig:pachner2}}
  \end{center}
\end{figure}

\begin{definition}
	\label{def:pg}
	The {\em Pachner graph} $\mathcal{P}$ of triangulated $2$-spheres is the graph whose nodes are triangulated $2$-spheres up to combinatorial isomorphism, with arcs between all pairs of triangulated $2$-spheres that can be transformed into isomorphic copies of each other by a single bistellar move.
\end{definition}

Note that it is a fundamental and well-known fact that the Pachner graph $\mathcal{P}$ of  triangulated $2$-spheres is connected (see for example \cite{Pachner87KonstrMethKombHomeo} for a much more general statement due to Pachner).

We denote the Pachner graph of all triangulated $2$-spheres with $n$ vertices by $\mathcal{P}_n$. Note that all edges in $\mathcal{P}_n$ correspond to edge flips.

The Pachner graph of $n$-vertex flag $2$-spheres is denoted by $\mathcal{F}_n$, the Pachner graph of $n$-vertex Hamiltonian $2$-spheres by $\mathcal{H}_n$, and the Pachner graph of  $n$-vertex stacked $2$-spheres by $\mathcal{S}_n$. Note that, naturally, all of these graphs are induced subgraphs in the Pachner graph $\mathcal{P}_n$ of $n$-vertex $2$-spheres. In particular, a priori it is not clear, whether or not any of them is connected. The following statement is due to work by Mori, Nakamoto and Ota.

\begin{theorem}[Theorem 5 of \cite{Komuro97DiameterOfFlipGraph}, Theorem 1 of \cite{Mori03DiagonalFlips}]
	\label{thm:mori}
For $n\geq 5$, the Pachner graph $\mathcal{H}_n$ is connected and of diameter at least $2n-15$ and at most $4n-20$.
\end{theorem}

In this article, we focus on structural properties of $\mathcal{F}_n $ and $\mathcal{S}_n$.

\section{The Pachner graph  \boldmath{$\mathcal{F}_n$} of \boldmath{$n$}-vertex flag $2$-spheres} \label{sec:flag}

In this section we prove that, for $n \geq 8$, the Pachner graph $\mathcal{F}_n $ of  $n$-vertex flag $2$-spheres  contains exactly two components, one of them consisting of the double cone $\Gamma_n$, the other one containing all other $n$-vertex flag $2$-spheres. Throughout this section we write $T \sim T'$ for two $n$-vertex flag $2$-spheres meaning that there exists a sequence of edge flips connecting $T$ and $T'$ preserving the flagness property at each step. We prove the following statement.

\begin{theorem}  \label{thm:pachnerGraph}
If $T$ and $T'$ are two $n$-vertex flag $2$-spheres distinct from $\Gamma_n$, then $T \sim T'$.
\end{theorem}

See Figures~\ref{fig:pachner8} to \ref{fig:pachner10} for illustrations of the Pachner graph $\mathcal{F}_n$ for $n \in \{ 8,9,10\}$.

\begin{figure}[!h]
  \begin{center}
    \includegraphics[width=0.2\textwidth]{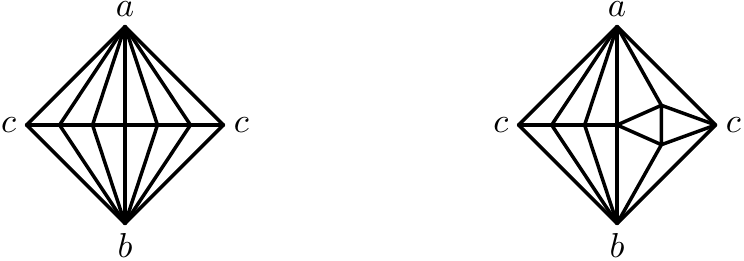}
    \caption{The Pachner graph $\mathcal{F}_8$ of $8$-vertex flag $2$-spheres. \label{fig:pachner8}}
  \end{center}
\end{figure}

\begin{figure}[!h]
  \begin{center}
    \includegraphics[width=.4\textwidth]{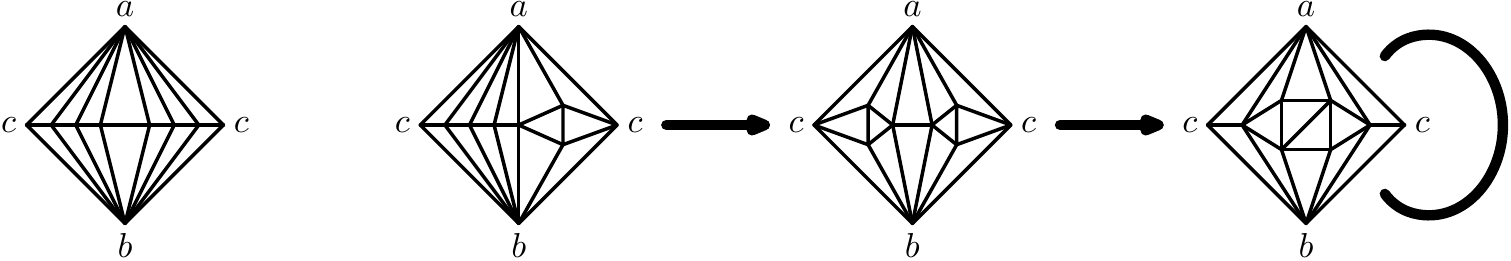}
    \caption{The Pachner graph $\mathcal{F}_9$ arranged
      left to right by decreasing separation indices, see \cite{Burton14SepIndex2Spheres}.
      \label{fig:pachner9}}
  \end{center}
\end{figure}

  \begin{figure}[!h]
    \begin{center}
      \includegraphics[width=\textwidth]{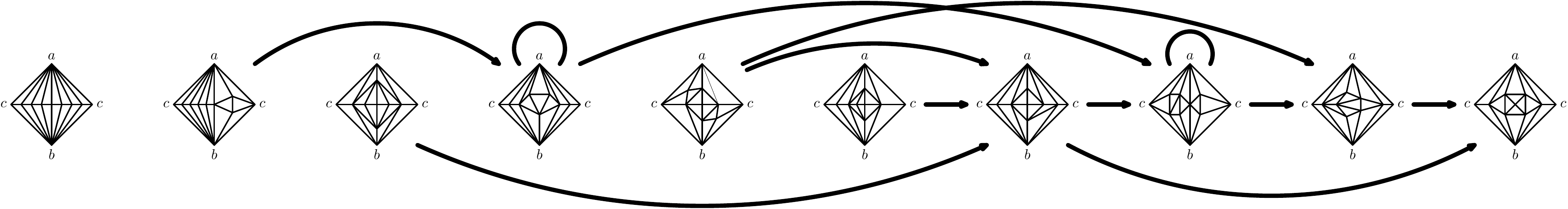}
      \caption{The Pachner graph $\mathcal{F}_{10}$ arranged
        left to right by decreasing separation indices, see \cite{Burton14SepIndex2Spheres}. \label{fig:pachner10}}
    \end{center}
  \end{figure}

The proof of Theorem~\ref{thm:pachnerGraph} relies on a number of lengthy and technical lemmas (Lemmas~\ref{lem:transport} to \ref{lem:organise}). We thus start by introducing all necessary terminology and a sketch of the proof, before proving all lemmas in detail.

\begin{definition}
	\label{def:proper}
    Let $T$ be a flag $2$-sphere. A subcomplex $Q$ of $T$ is called a {\em quadrilateral} if it is a triangulated disc and its boundary is a $4$-cycle. A quadrilateral $Q$ in $T$ with boundary $a\mbox{-}b\mbox{-}c\mbox{-}d\mbox{-}a$ is called {\em proper}, if $a\mbox{-}b\mbox{-}c\mbox{-}d\mbox{-}a$ is an induced cycle in $T$ and $\deg_T(a), \deg_T(b), \deg_T(c), \deg_T(d) \geq 5$. Since the boundary is an induced cycle, a proper quadrilateral contains at least one interior vertex. A quadrilateral $Q$ in $T$ is called {\em ordered}, if it contains an interior vertex, and all of its interior vertices are of degree four. Since an ordered quadrilateral is a subcomplex of a flag $2$-sphere, it follows that all the interior vertices lie on a path connecting diagonally opposite vertices of $Q$. We call this path a {\em diagonal path}, or just a {\em diagonal of $Q$}.
\end{definition}

\begin{definition}
	\label{def:An}
	For $n\geq 7$, let $A_n$ in $\mathcal{F}_n$ be as in Figure~\ref{fig:An}(b). Note that $A_7 = \Gamma_7$, $A_n \neq \Gamma_n$ for $n \geq 8$, and that $A_n$ is a vertex of degree one in $\mathcal{F}_n$ for $n\geq 9$.

	For $k \geq 3$, let $\mathcal{Q}_k$ be the triangulated quadrilateral with $k$ interior vertices shown in Figure~\ref{fig:An}(c). The path $a_0\mbox{-}a_1\mbox{-}\cdots\mbox{-}a_k$ is said to be the {\em diagonal path} of $\mathcal{Q}_k$.

 \begin{figure}[!h]
   \centerline{
      \includegraphics[height=3.8cm]{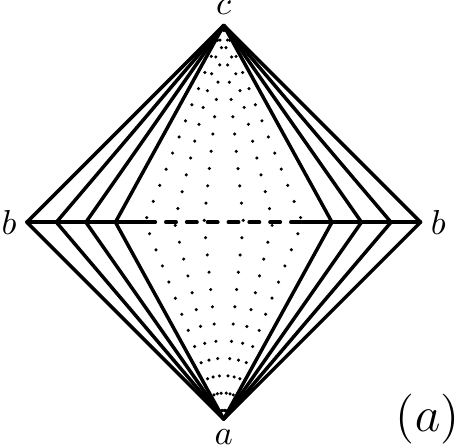} \hspace{1cm}
     \includegraphics[height=3.8cm]{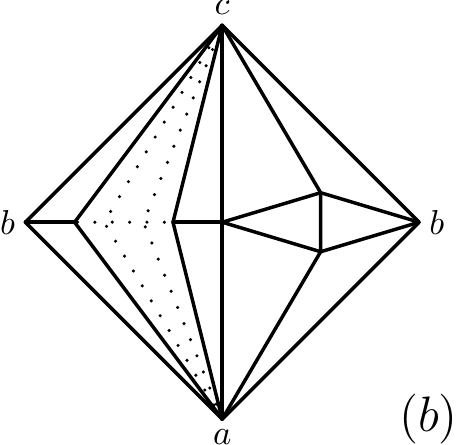} \hspace{1cm}
    \includegraphics[height=3.8cm]{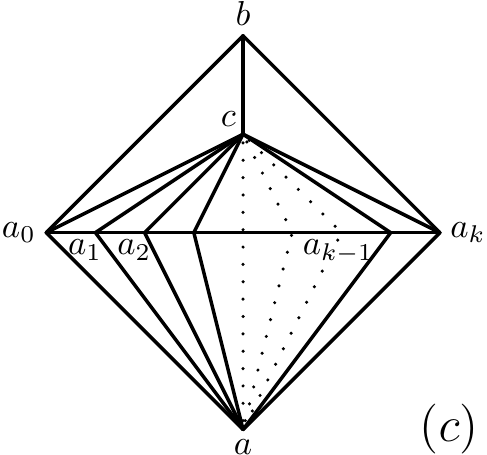}}
    \caption{(a) Double cone $\Gamma_n$ over the $(n-2)$-gon. (b) Target $n$-vertex flag $2$-sphere $A_n$. (c) Quadrilateral $\mathcal{Q}_k$ with boundary vertices $a_0, a, a_k, b$ and interior vertices $c$, $a_1, \ldots , a_{k-1}$.
       \label{fig:An}}
  \end{figure}
\end{definition}

We prove Theorem~\ref{thm:pachnerGraph} by showing that $T \sim A_n$, for any $n$-vertex flag $2$-sphere $T$ distinct from the double cone. For this, we split $T$ (or a slight variation thereof) along an induced $4$-cycle into two triangulated quadrilaterals $Q$ and $R$ using Lemma~\ref{lem:induced}. We then use Lemma~\ref{lem:organise} to turn all interior vertices of both $Q$ and $R$ into vertices of degree four. Finally, we use Lemma~\ref{lem:transport} to transport excess internal vertices from $R$ to $Q$ (or vice versa), until we obtain $A_n$.

The main difficulty in the above procedure is to prove Lemma~\ref{lem:organise}. For this we need Lemma~\ref{lem:red2}, which allows us to merge two smaller triangulated quadrilaterals, and Lemma~\ref{lem:Qk}, which allows us to resolve a pathological class of triangulations of the quadrilateral (triangulation $\mathcal{Q}_k$, shown in Figure~\ref{fig:An}(c)). In addition, all of Lemmas~\ref{lem:Qk}, \ref{lem:red2} and \ref{lem:organise} need Lemma~\ref{lem:transport} to transport internal vertices from one quadrilateral to another.

For a more precise but less descriptive outline, see the proof of Theorem~\ref{thm:pachnerGraph} at the end of this section.

\begin{lemma}[Transport Lemma]  \label{lem:transport}
Let $T$ be a flag $2$-sphere containing two ordered quadrilaterals $\alpha$ and $\beta$ with disjoint interiors, but a common boundary edge $vw$. Furthermore, let $k \geq 2$ ($\ell \geq 1$) be the number of interior vertices of $\alpha$ (resp., $\beta$), and let $v$ and $w$ satisfy one of the following conditions:
  \begin{enumerate}[(1)]
    \item $\deg_T(w) \geq 5$, and the diagonal paths of $\alpha$ and $\beta$ intersect in $w$;
    \item $\deg_T(v) \geq 5$, $\deg_T(w) \geq 6$,
      the diagonal path of $\alpha$ intersects $v$, and the diagonal path of $\beta$
      intersects $w$.
  \end{enumerate}

Then there exists a flag $2$-sphere $T'$ such that (i) $T \sim T'$, (ii) $T'$ contains two ordered quadrilaterals $\alpha'$ and $\beta'$, (iii) $T' = (T \setminus \{ \alpha , \beta \} ) \cup \{ \alpha' , \beta' \}$, (iv) $vw$ is a common edge of $\alpha'$ and $\beta'$ in $T'$, and (v) the number of interior vertices of $\alpha '$ is $k-1$, and the number of interior vertices of $\beta' $ is $\ell+1$.
\end{lemma}

Lemma~\ref{lem:transport} gives precise conditions on when exactly we can
``transport'' an interior vertex of an ordered quadrilateral of $T$ into an adjacent ordered quadrilateral without changing anything else in $T$. Both Condition {\em (1)} and {\em (2)} for Lemma~\ref{lem:transport} are necessarily fulfilled as soon as $\alpha$ and $\beta$ only share one edge. If $\alpha$ and $\beta$ share two edges, the situation is different: In Condition {\em (1)} we can then have $\deg_T (w) =4$, in Condition {\em (2)} and for $k=2$ and $\ell=1$ we can have both $\deg_T(v) = 4$ and $\deg_T(w) = 5$.

\begin{proof}
Each ordered quadrilateral of $T$ must be subdivided by a diagonal path containing all of its interior vertices all of which are of degree four. Hence, up to exchanging the roles of $v$ and $w$, there are two possible initial configurations to consider: The diagonal paths of $\alpha$ and $\beta$ either meet, or one ends in $v$ and the other in $w$. The former corresponds to Condition {\em (1)} of the Lemma, the latter one to Condition {\em (2)}

  \begin{figure}[!h]
    \begin{center}
      \includegraphics[width=.225\textwidth]{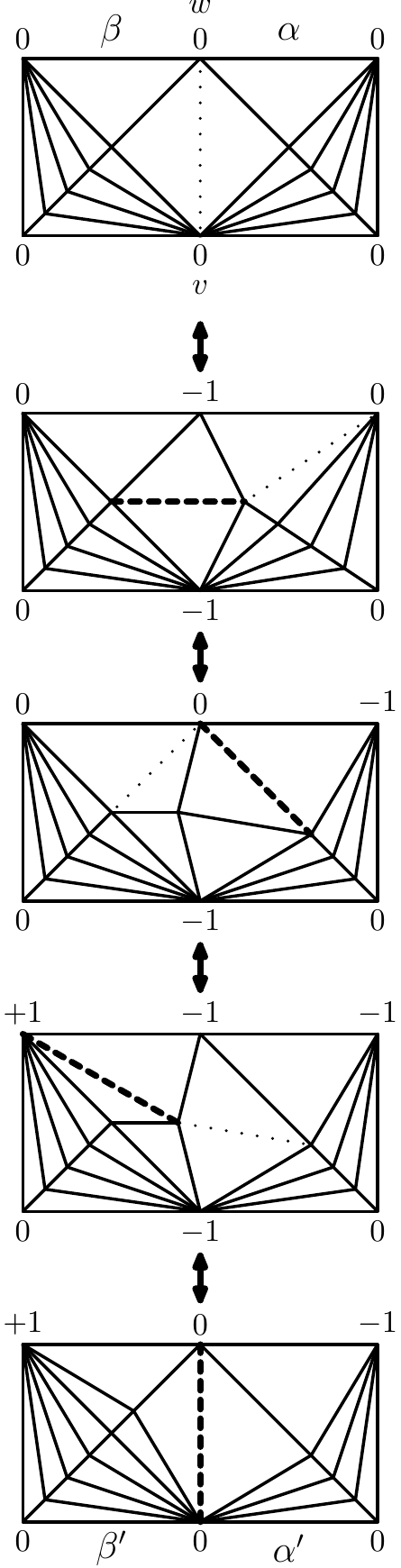} \hspace{2.5cm}
      \includegraphics[width=.225\textwidth]{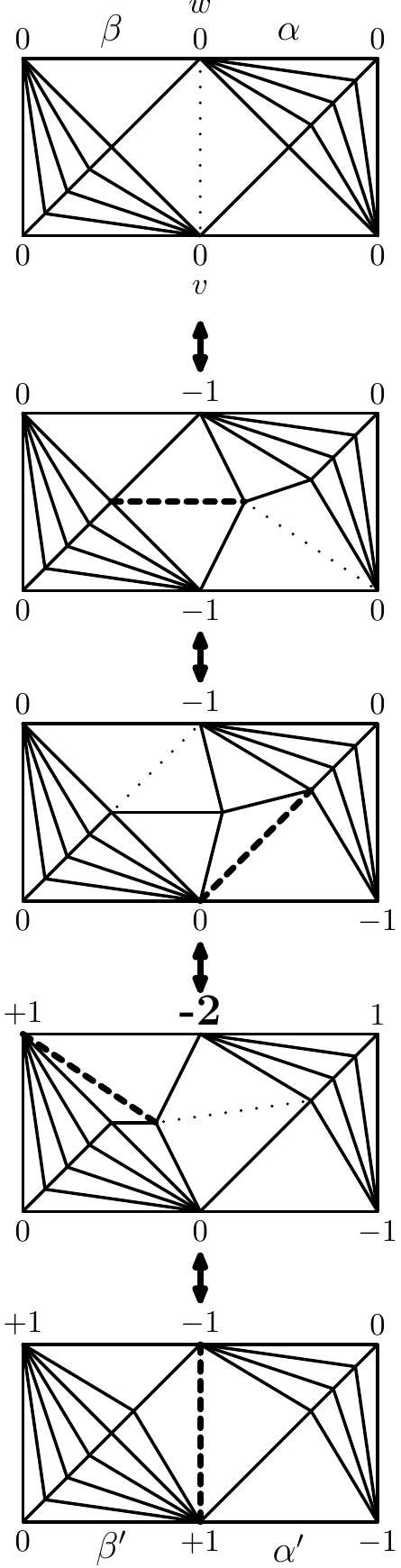}
      \caption{Transport Lemma. Left: sequence of edge flips for intersecting diagonal paths (Cond. {\em (1)}).
        Right: sequence of edge flips for diagonal paths ending in $v$ and $w$ respectively (Cond. {\em (2)}). \label{fig:flips}}
    \end{center}
  \end{figure}

\noindent {\bf Condition {\em (1)}} The diagonal paths of $\alpha$ and $\beta$ meet in $w$. In this case, the sequence of flips transforming $T$ to $T'$ is shown in Figure~\ref{fig:flips} on the left hand side (top to bottom). The dotted edge denotes the edge to be flipped next, the dashed line denotes the newly inserted edge. The integer next to a vertex indicates the change of the respective vertex degree with respect to the initial vertex degree.

Throughout this edge flip sequence the degrees of $w$, $v$ and the upper left vertex of $\alpha$ are, at some point, decreased to the initial degree minus one. The degrees of all other boundary vertices are never decreased below the initial degree. Since all three vertices of the former group are initially of degree at least five ($w$ by assumption and the other two by the flagness of $T$), the flagness condition is preserved in each step. The preconditions of the lemma ensure that no $3$-cycle is introduced in the first flip, the edges introduced by flip two and three end in the interior of $\alpha \cup \beta$ and hence cannot introduce a new $3$-cycle, and the last flip re-introduces the edge removed by the first flip.

\noindent {\bf Condition {\em (2)}} $\alpha$ and $\beta$ have diagonal paths ending in $v$ and $w$ respectively. To comply with the labelling of the statement of the lemma, let the diagonal of $\alpha$ intersect with $v$ and the diagonal of $\beta$ intersect with $w$. The sequence of edge flips transforming $T$ to $T'$ in this case is shown in Figure~\ref{fig:flips} on the right hand side (top to bottom). Meaning of dotted and dashed lines as well as integers next to vertices as in Condition {\em (1)}

Note that, in this procedure, only the degree of $w$ is, at one stage, decreased to the initial degree minus two. In addition, $v$ and the lower left  vertex of $\alpha$ are, at some point, decreased to the initial degree minus one. The degrees of all other boundary vertices are never decreased below the initial degree. By assumption, $w$ is of initial degree at least six and $v$ is of initial degree five. Again, the other vertex of $\alpha$ not containing the diagonal must be of initial degree at least five by the flagness of $T$. It follows that the flagness condition is preserved in each step. Again, no $3$-cycle is introduced by the flip sequence for reasons analogous to the ones described in the previous case.
\end{proof}

\begin{lemma} \label{lem:Qk}
Let $T$ be an $n$-vertex flag $2$-sphere, $n\geq 8$, with induced $4$-cycle $a\mbox{-}a_0\mbox{-}b\mbox{-} a_k\mbox{-}a$ bounding $\mathcal{Q}_k$. Then either $T=\Gamma_n$, $T \sim A_n$, or there exists an $n$-vertex flag $2$-sphere $T'$ with $T \sim T'$, such that (i) $a\mbox{-}a_0\mbox{-}b\mbox{-}a_k\mbox{-}a$ is an induced $4$-cycle in $T'$ bounding an ordered quadrilateral $Q$, and (ii) $T \setminus \mathcal{Q}_k = T' \setminus Q$.
\end{lemma}

\begin{proof}
  We use the notation for $\mathcal{Q}_k$ as introduced in Figure~\ref{fig:An}(c) and in accordance with the vertex labels of the induced $4$-cycle $a\mbox{-}a_0\mbox{-}b\mbox{-} a_k\mbox{-}a$ bounding $\mathcal{Q}_k$.
	
  \medskip
	
  \noindent {\bf Case $k=3$}: Refer to Figure~\ref{fig:Q3}(a). Consider the two triangles $a_0 a x_1, a_3 a x'_1 \in F(T)$ outside but adjacent to $\mathcal{Q}_3$. If $x_1 = x'_1$ (i.e., $\deg_T (a) =5$) consider triangles $a_0 x_i x_{i+1}, a_3 x_i x'_{i+1} \in F(T)$, $i \geq 1$, until either $x_{\ell+1} \neq x'_{\ell+1}$, that is, $\deg_T (x_{\ell}) \geq 5$, or $x'_{\ell} = x_{\ell} = b$.
	
The case $x'_{1} = x_{1} = b$ is not possible because $a\mbox{-}a_0\mbox{-}b\mbox{-} a_k\mbox{-}a$ is induced (and because $n \geq 8$). If $x'_{\ell} = x_{\ell} = b$, $\ell \geq 2$, $T$ must be isomorphic to $A_{\ell + 6}$ and we are done. Otherwise, consider the two triangles $a_0 x_{\ell} x_{\ell+1}$ and $a_3 x_{\ell} x'_{\ell+1}$, $x'_{\ell+1} \neq x_{\ell+1}$. Neither $a_0 x'_{\ell+1}$ nor $a_{3} x_{\ell+1}$ can be edges of $T$ since otherwise there are induced $3$-cycles $a_0\mbox{-}x'_{\ell+1}\mbox{-}x_{\ell}\mbox{-}a_0$ or $a_3\mbox{-}x_{\ell+1}\mbox{-}x_{\ell}\mbox{-}a_3$.

Keeping these observations in mind, we perform edge flip $a_0 x_{\ell} \mapsto x_{\ell+1} x_{\ell -1}$ (see Figure~\ref{fig:Q3}(b)), followed by edge flips $a_0 x_{\ell-1} \mapsto x_{\ell+1} x_{\ell -2}$, etc. all the way down to $a_0 a \mapsto x_{\ell+1} a_1$ (see Figure~\ref{fig:Q3}(c)). For each of them we have that, since $a_3 x_{\ell+1}$ is not an edge, $a_3 \mbox{-}x_i\mbox{-}x_{\ell+1}\mbox{-}a_3$ is not a $3$-cycle of $T$.
	
It follows that we can perform flips $a_1 c \mapsto a_0 a_2$ (Figure~\ref{fig:Q3}(d)) and $a a_2 \mapsto a_1 a_3$ (Figure~\ref{fig:Q3}(e)), followed by the initial sequence of edge flips in reverse, i.e., $x_{\ell+1} a_1 \mapsto a_0 a$, $x_{\ell+1} a \mapsto a_0 x_1$, $x_{\ell+1} x_1 \mapsto a_0 x_2$, all the way up to $x_{\ell+1} x_{\ell-1} \mapsto a_0 x_{\ell}$ (Figure~\ref{fig:Q3}(f)). Observe that now all vertices inside $\mathcal{Q}_3$ are of degree four and outside $\mathcal{Q}_3$ the triangulation is unchanged. This proves the result for $k=3$.
		
  \begin{figure}[htbp]
    \centerline{\includegraphics[width=.91\textwidth]{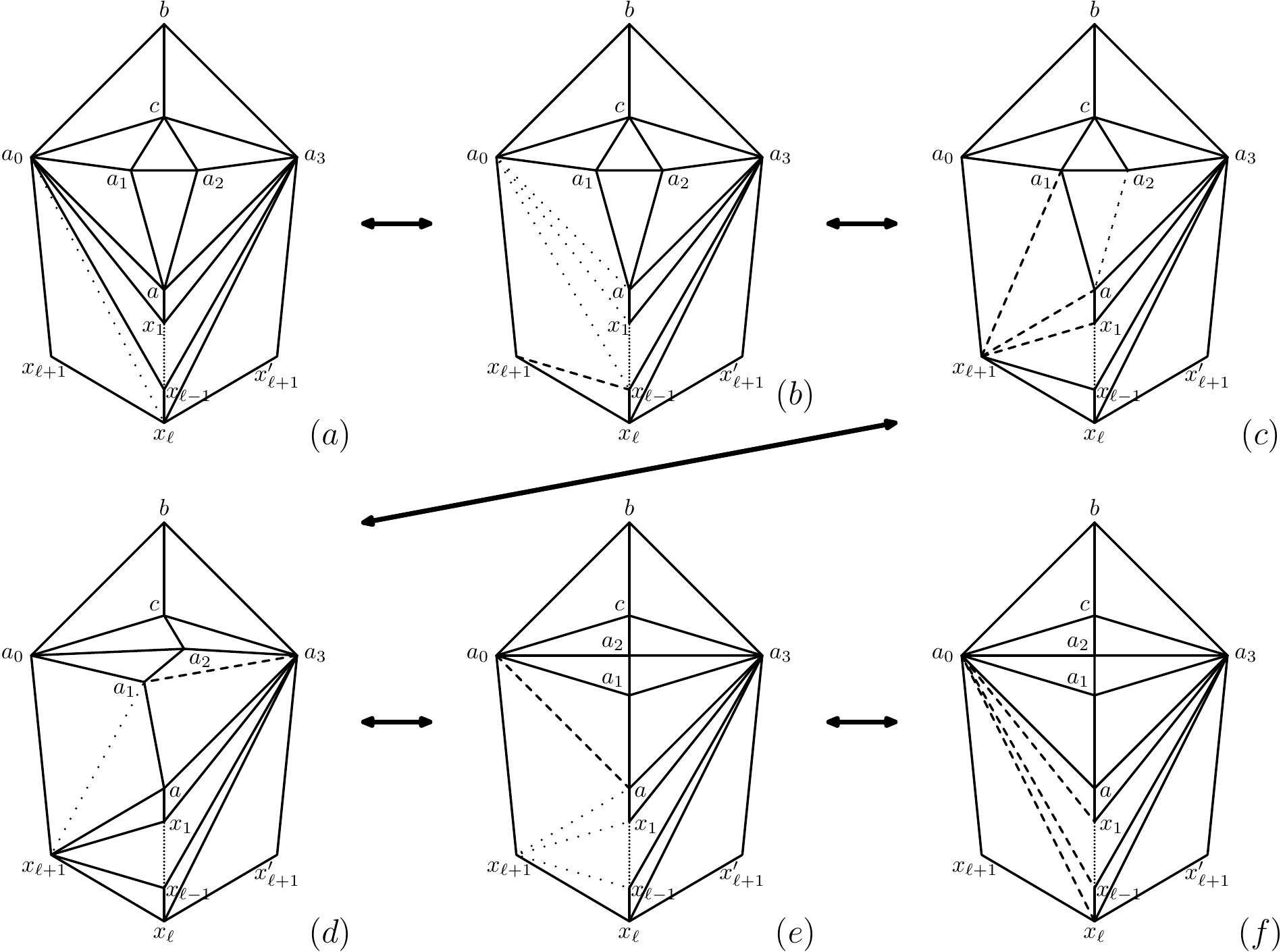}}
    \caption{Resolving $\mathcal{Q}_3$ into a quadrilateral with three interior vertices of degree four.\label{fig:Q3}}
  \end{figure}		

  \begin{figure}[htbp]
    \centerline{\includegraphics[width=.91\textwidth]{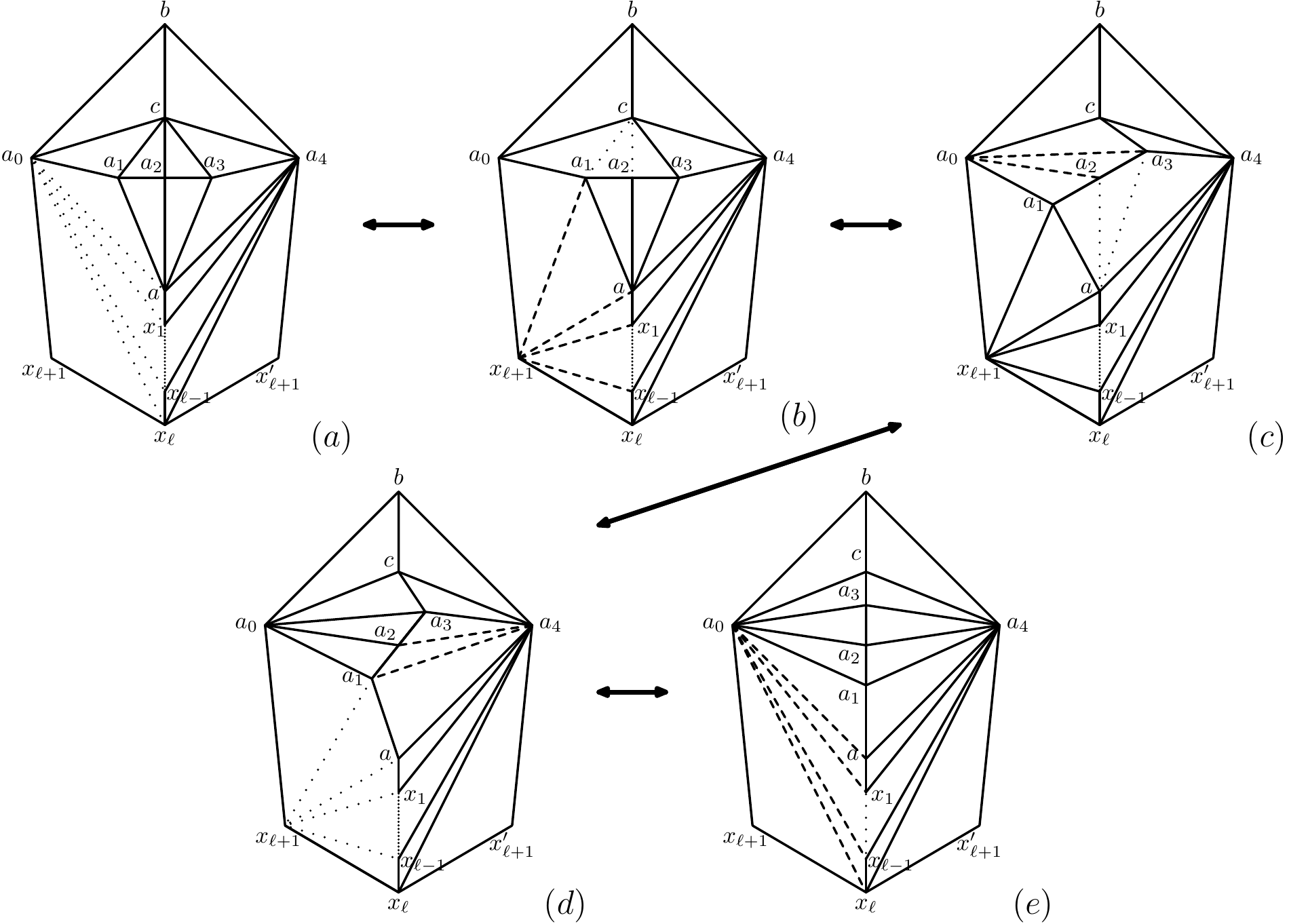}}
    \caption{Resolving $\mathcal{Q}_4$ into a quadrilateral with four interior vertices of degree four. \label{fig:Q4}}
  \end{figure}
	
\noindent {\bf Case $k = 4$}: Refer to Figure~\ref{fig:Q4}(a). The case $k=4$ is very similar to the case $k=3$. Again, the case $x'_{1} = x_{1} = b$ is not possible because $a\mbox{-}a_0\mbox{-}b\mbox{-} a_k\mbox{-}a$ is induced. If $x'_{\ell} = x_{\ell} = b$ for $\ell \geq 2$, $T$ decomposes into two ordered proper quadrilaterals along induced $4$-cycle $a_0\mbox{-}a\mbox{-}a_4\mbox{-}c\mbox{-}a_0$ to which we can apply Lemma~\ref{lem:transport}: The ordered proper quadrilateral contained in $\mathcal{Q}_4$, the rest of $T$, $a$ and $a_0$ take the roles of $\alpha$, $\beta$, $w$ and $v$. The diagonals are disjoint, $\deg_T (a) = 6$ and $\deg_T (a_0) \geq 5$. In particular, Condition {\em (2)} is satisfied with $k=3$ and $\ell \geq 1$ and we can transport $a_1$ or $a_3$ away from its quadrilateral to conclude that $T\sim A_n$, $n \geq 8$.

If $x_{\ell+1} \neq x'_{\ell+1}$ for some $\ell \geq 0$ we perform a sequence of edge flips similar to the one in the case $k = 3$ above. More precisely, the initial set of flips (Figure~\ref{fig:Q4}(b)) and the final set of flips (Figure~\ref{fig:Q4}(e)) are identical with the initial two and the final two steps of case $k=3$. Once flip $a_0 a \mapsto x_{\ell+1} a_1$ is performed, we can perform $a_1 c \mapsto a_0 a_2$ and $a_2 c \mapsto a_0 a_3$ (Figure~\ref{fig:Q4}(c)), followed by $a_3 a \mapsto a_2 a_4$ and $a_2 a \mapsto a_1 a_4$ (Figure~\ref{fig:Q4}(d)).
	
\noindent {\bf Case $k > 4$}: Refer to Figure~\ref{fig:Qn}(a). From $k > 4$ it follows that $n \geq 10$. Moreover, $a_0 \neq a_k$, and $a_0 a_k$ is a non-edge of $T$ since $a_0\mbox{-}a\mbox{-}a_k\mbox{-}b\mbox{-}a_0$ is an induced $4$-cycle. We start by performing flips $a_0 c \mapsto b a_1$, $a_1 c \mapsto b a_2$, all the way to $a_{k-4} c \mapsto b a_{k-3}$ (see Figure~\ref{fig:Qn}(b)). The resulting quadrilateral splits into two parts. One with only degree four interior vertices (at least one), the other one being isomorphic to $\mathcal{Q}_3$ with diagonal path going from $a_{k-3}$ to $a_k$ (see Figure~\ref{fig:Qn}(c) for a re-arranged version of the top centre quadrilateral emphasizing this fact).

  \begin{figure}[hbt]
    \centerline{\includegraphics[width=0.95\textwidth]{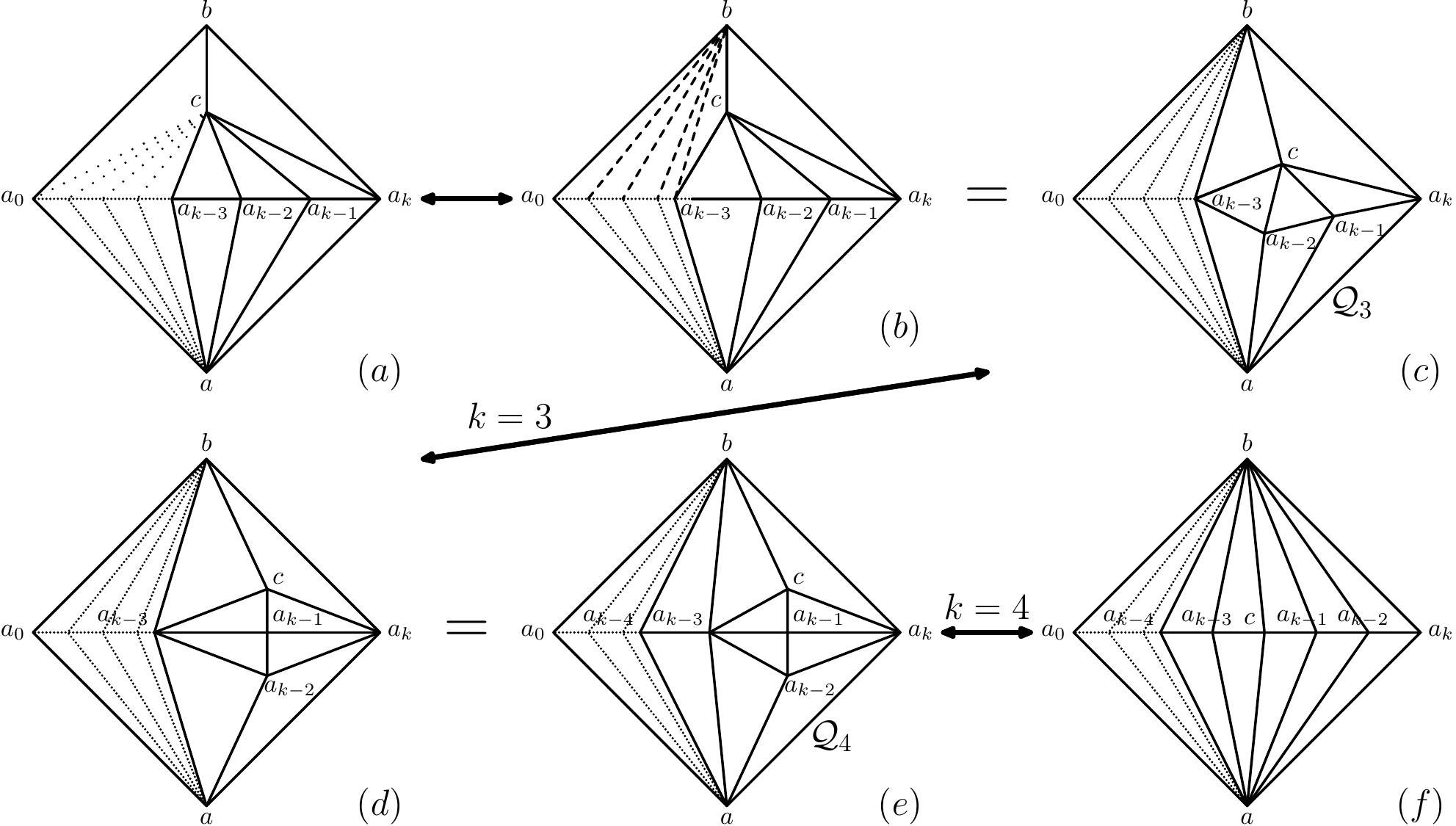}}
    \caption{Resolving $\mathcal{Q}_k$, $k > 4$, into a quadrilateral with $k$ interior vertices of degree four. \label{fig:Qn}}
  \end{figure}	

Use the case $k=3$ to turn $\mathcal{Q}_3$ into a quadrilateral containing only interior vertices of degree four with the diagonal path running from $a$ to $b$ (see Figure~\ref{fig:Qn}(d)). Since $k>4$, the overall quadrilateral again splits into two parts, one with only degree four interior vertices (possibly none), the other one being isomorphic to $\mathcal{Q}_4$ with diagonal path going from $a$ to $b$ (see Figure~\ref{fig:Qn}(e) for a re-arranged version of the bottom left quadrilateral emphasizing this fact). Use the case $k=4$ to either conclude that $T \sim A_n$, or to turn $\mathcal{Q}_4$ into a quadrilateral containing only degree four interior vertices and diagonal running from $a_{k-4}$ to $a_k$. In the latter case the overall quadrilateral now only has interior vertices of degree four which proves the lemma (see Figure~\ref{fig:Qn}(f)).
\end{proof}

\begin{lemma}[Merge Lemma]   \label{lem:red2}
Let $T$ be an $n$-vertex flag $2$-sphere containing two ordered quadrilaterals $\alpha$ and $\beta$ with disjoint interiors, but common outer edges $uv$ and $uw$. Then either $T = \Gamma_n$, $T \sim A_n$, or $T \sim T'$ where $T'$ has an ordered quadrilateral $\gamma$ with boundary $\partial (\alpha \cup \beta)$ and $T' = ( T \setminus \{ \alpha, \beta \}) \cup \{ \gamma \}$.
\end{lemma}

\begin{proof}
  We have four cases for the initial configuration of $\alpha$ and $\beta$ emerging from the different possible relative orientations of the diagonal paths of $\alpha$ and $\beta$, see Figure~\ref{fig:initConfig}.

  \begin{figure}[b]
    \begin{center}
      \includegraphics[width=.9\textwidth]{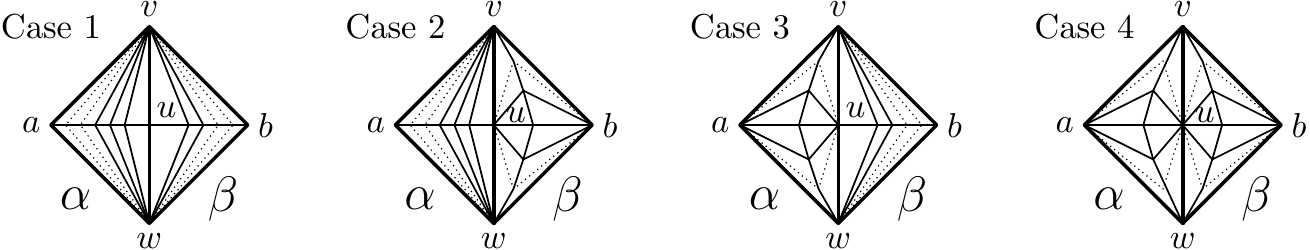}
    \end{center}
    \caption{The four initial configurations for $\alpha$ and $\beta$ in the proof of Lemma~\ref{lem:red2}. \label{fig:initConfig}}
  \end{figure}

\medskip

\noindent {\bf Case 1:} If $a=b$ then $\alpha \cup \beta = T$. In this case, $T = \Gamma_n$ with cone apices $v$ and $w$ and we are done.

If $a \neq b$ we can merge $\alpha$ and $\beta$ into one larger ordered quadrilateral with boundary $\partial (\alpha \cup \beta)$.

\begin{figure}[hbt]
  \begin{center}
    \includegraphics[width=\textwidth]{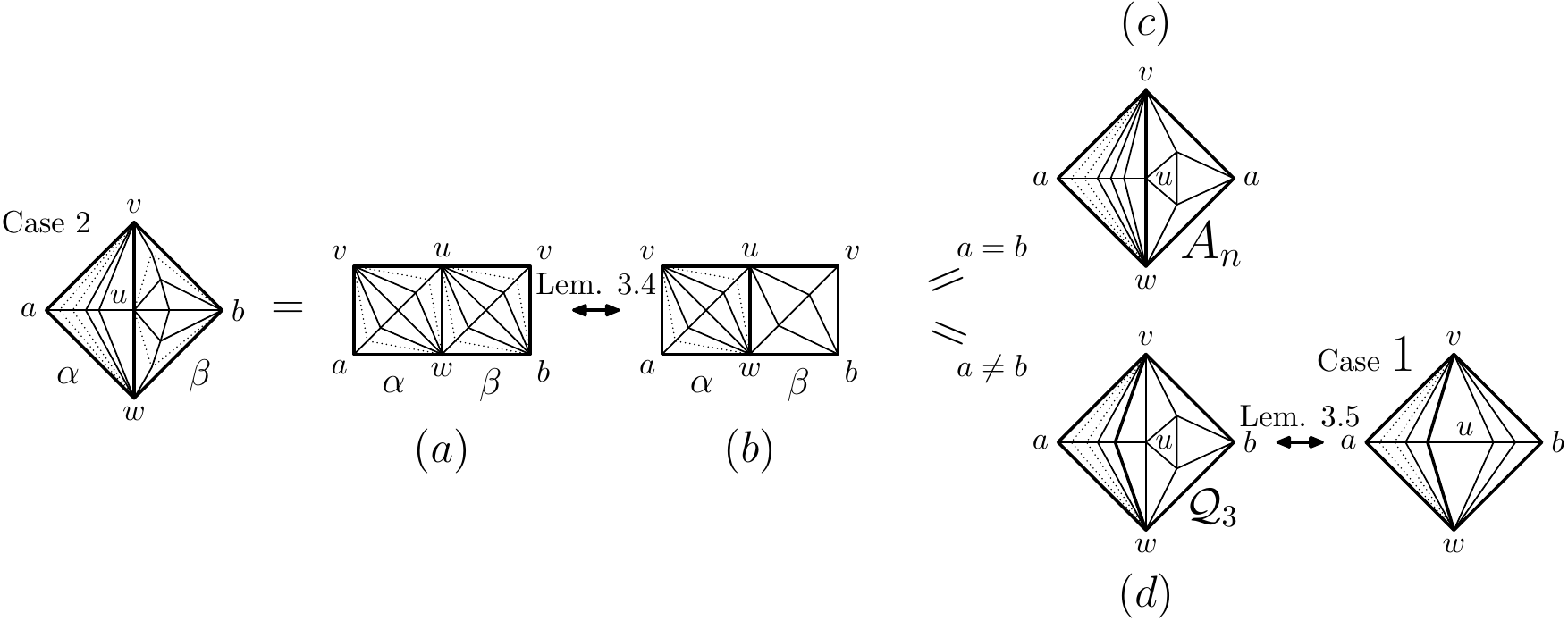}
    \caption{Transporting vertices in Case 2: (a) Case 2 redrawn after cutting along edge $uv$. (b) After transporting interior vertices away from $\beta$ (Lemma~\ref{lem:transport}). (c) Case $a=b$ yields $A_n$. (d) Case $a\neq b$ yields $\mathcal{Q}_3$. In the latter case apply Lemma~\ref{lem:Qk} to fall back to Case 1. \label{fig:case2}}
  \end{center}
\end{figure}

\noindent {\bf Case 2:} Refer to Figure~\ref{fig:case2}. As before, if $a=b$ then $\alpha \cup \beta = T$. If, in this case, $\beta$ contains only one interior vertex, then we have $\Gamma_n$ with cone apices $v$ and $w$ and we are done. If $\alpha$ contains only one interior vertex, both $v$ and $w$ are of degree four, and we have $\Gamma_n$ with cone apices $u$ and $a=b$. 

Thus, we can assume both $\alpha$ and $\beta$ have at least two interior vertices. In this case, we iteratively apply Lemma~\ref{lem:transport} to transport interior vertices from $\beta$ to $\alpha$ across edge $uw$ until $u$ is of degree five (Figure~\ref{fig:case2}(b)) and we obtain $A_n$ (Figure~\ref{fig:case2}(c)).

If $a \neq b$, we, again, apply Lemma~\ref{lem:transport} to transport interior vertices from $\beta$ to $\alpha$ across edge $uw$ until $u$ is of degree five (Figure~\ref{fig:case2}(b)). The quadrilateral $\beta$ together with the two rightmost triangles of $\alpha$ now form a quadrilateral isomorphic to $\mathcal{Q}_3$ with diagonal path from $v$ to $w$ (see Figure~\ref{fig:case2}(d)). This can be resolved into a quadrilateral with interior vertices all of degree four and diagonal intersecting $b$ (note that $a$ is of degree greater than four and thus (i) the preconditions of Lemma~\ref{lem:Qk} are satisfied and (ii) we can always resolve $\mathcal{Q}_3$ in this case) and we are back to Case~$1$.

\medskip

\noindent {\bf Case 3:} This is completely analogous to Case 2.

\medskip

\noindent {\bf Case 4:} Again, if $a=b$ then $\alpha \cup \beta = T$, and $T$ is equal to $\Gamma_n$ with cone apices $u$ and $a=b$.

Hence, let $a \neq b$. If $\alpha$ contains only a single interior vertex we fall back to Case 2, if $\beta$ contains only a single vertex we fall back to Case 3. Thus we can assume both $\alpha$ and $\beta$ have at least two interior vertices. In this case, $\deg_T(u) \geq 6$, $\deg_T(v), \deg_T(w) \geq 5$, and we apply Lemma~\ref{lem:transport} to transport vertices from $\alpha$ to $\beta$ until $\alpha$ contains only a single interior vertex. Then we proceed with Case 2.
\end{proof}

\begin{lemma} \label{lem:induced}
For $n\geq 8$, let $T \in \mathcal{F}_n\setminus\{\Gamma_n\}$. Then there exists $T' \in \mathcal{F}_n\setminus\{\Gamma_n\}$ with $T \sim T'$, and $a,b,c,d \in V(T')$ such that (i) $a\mbox{-}b\mbox{-}c\mbox{-}d\mbox{-}a$ is an induced $4$-cycle, and (ii) $\deg_{T'}(a)$, $\deg_{T'} (b)$, $\deg_{T'} (c)$, $\deg_{T'} (d) \geq 5$. In particular, $T'$ splits into two proper quadrilaterals $Q$ and $R$ both bounded by $a\mbox{-}b\mbox{-}c\mbox{-}d\mbox{-}a$.
\end{lemma}

\begin{proof}
If $T$ contains a vertex $v$ of degree four, then, by the flagness of $T$, the link of $v$ is an induced $4$-cycle, say $a\mbox{-}b\mbox{-}c\mbox{-}d\mbox{-}a$. If any of these vertices, say $a$, is of degree four, then, since $n\geq 8$, the boundary of the union of the stars $v$ and $a$ is an induced $4$-cycle. Moreover, $b$ and $d$ are of degree at least five. Iterating this process either yields an induced $4$-cycle $x\mbox{-}b\mbox{-}c\mbox{-}d\mbox{-}x$, for some vertex $x$ of $T$ of degree at least five, or $x=c$, and $T$ is isomorphic to $\Gamma_n$, a contradiction. Hence, assume $\deg_T (x) \geq 5$, and thus $x \neq c$. If the degree of $c$ is $4$, consider the union of the quadrilateral containing $v$ and bounded by $x\mbox{-}b\mbox{-}c\mbox{-}d\mbox{-}x$ and the star of vertex $c$. As before, iterate this procedure until we obtain an induced $4$-cycle $x\mbox{-}b\mbox{-}y\mbox{-}d\mbox{-}x$ in $T$ (possibly $y=c$) with $x$ and $y$ necessarily distinct and both of degree at least five (note that $x=y$ implies $T$ isomorphic to $\Gamma_n$ and thus $\deg_T(x) = 4$, a contradiction).
	
Since $T$ is flag, it cannot contain a vertex of degree three. If, in addition, $T$ does not contain a vertex of degree four, then $T$ must contain a vertex $w$ of degree five (this is a consequence of Euler's formula which implies that the average vertex degree of a triangulated $2$-sphere must be less than six). Let $a u w$ and $b u w$ be two adjacent triangles in the star of $w$. If $a$ and $b$ have a common neighbour $x$ distinct from $w$ and $u$, then $x\mbox{-}a\mbox{-}w\mbox{-}b\mbox{-}x$ is an induced $4$-cycle, and we are done since $T$ has no vertex of degree four. Otherwise the flip $uw \mapsto ab$ yields a flag $2$-sphere in which $w$ has degree four. Now the link of $w$ is an induced $4$-cycle with all four vertices being of degree at least five.
\end{proof}

\begin{lemma} 	\label{lem:organise}
Let $T$ be an $n$-vertex flag $2$-sphere which splits into two proper quadrilaterals $Q$ and $R$ along an induced $4$-cycle $a\mbox{-}c\mbox{-}b\mbox{-}d\mbox{-}a$. Then there exists an $n$-vertex flag $2$-sphere $T'$ with $T \sim T'$, such that $T' = Q' \cup R$, and the interior of $Q'$ contains only degree four vertices.
\end{lemma}

Note that, in $T'$, neither $Q'$ nor $R$ need to be proper quadrilaterals. However, both $Q'$ and $R$ contain interior vertices. In particular, each of $a$, $b$, $c$, and $d$ is contained in at least two triangles of both $Q'$ and $R$.
We deal with this issue separately whenever we need to, namely in the proof of Theorem~\ref{thm:pachnerGraph}.

\begin{proof}
  We prove this statement by induction on the number $k$ of interior vertices in $Q$. First note that $k > 0$, and that the statement is true for $k \leq 2$.
	
  Let $a\mbox{-}c\mbox{-}b\mbox{-}d\mbox{-}a$ be the boundary of a quadrilateral $Q$ in $T$ with $k \geq 3$ interior vertices, such that $\deg_{T} (a),\deg_{T} (b), \deg_{T} (c),\deg_{T} (d) \geq 5$. Since $a\mbox{-}c\mbox{-}b\mbox{-}d\mbox{-}a$ is induced, $ab$ and $cd$ cannot be edges of~$T$.
	
  \medskip

  \noindent {\bf Claim}: There exist a triangulation $T'$ with $T \sim T'$, such that $T' = Q' \cup R$, and in the interior of $Q'$ either $a$ and $b$ or $c$ and $d$ have at least one common neighbour.
	
  \medskip
  We first complete the proof of the lemma assuming the claim is true. This is then followed by a proof of the claim. We can thus assume that we have an $n$-vertex flag $2$-sphere $T'$, $T \sim T'$, such that either $a$ and $b$ or $c$ and $d$ have at least one common neighbour in $Q'$.

  Assume that there exist at least one common neighbour of $a$ and $b$ (the case that $c$ and $d$ have at least one common neighbour is completely analogous). If all such neighbours are of degree four, all interior vertices must be neighbours of $a$ and $b$ of degree four and we are done. Otherwise, choose a common neighbour $e$ of degree at least five, and split $Q'$ into two smaller quadrilaterals $Q_1$ and $Q_2$ with boundaries $e\mbox{-}a\mbox{-}c\mbox{-}b\mbox{-}e$ and $e\mbox{-}a\mbox{-}d\mbox{-}b\mbox{-}e$ respectively. Without loss of generality, let $Q_2$ be the quadrilateral with at least three triangles containing $e$.

  If $Q_1$ has interior vertices, use the induction hypothesis to obtain a $2$-sphere $T''$, $T' \sim T''$, in which $Q_1$ is transformed into a quadrilateral $Q'_1$ with boundary $e\mbox{-}a\mbox{-}c\mbox{-}b\mbox{-}e$, $T'\setminus Q_1 = T''\setminus Q'_1$, and in which all interior vertices of $Q'_1$ have degree four. In $T''$ vertex $d$ is still of degree at least five, vertices $a$ and $b$ must be of degree at least six, and vertex $e$ must be of degree at least five since at least three triangles containing $e$ are outside $Q'_1$. In particular, $Q_2$ is proper and we can apply the induction hypothesis to $Q_2$ to obtain a triangulated $2$-sphere $T'''$ with two ordered quadrilaterals $Q'_1$ and $Q'_2$ joined along two adjacent edges. Use Lemma~\ref{lem:red2} to merge both quadrilaterals, or conclude that $T \sim A_n$. We have that $T \not \sim \Gamma_n$, since $\Gamma_n$ does not split into to proper quadrilaterals, as required by the statement of Lemma~\ref{lem:organise}.
	
  Hence, without loss of generality let $Q_1$ be without interior vertices. Use the induction hypothesis to transform $Q_2$ into $Q'_2$ with only degree four vertices inside. Now either $e$ is of degree four, all interior vertices of $Q' = Q_1 \cup Q'_2$ are of degree four, and we are done. Or $Q'$ is isomorphic to $\mathcal{Q}_{k}$ and, by Lemma~\ref{lem:Qk}, can be transformed into a quadrilateral containing only degree four vertices (or $T \sim A_n $), and again we are done.

  \medskip

  \noindent {\bf Proof of the claim}: Refer to Figure~\ref{fig:quads}. In the following procedure we always denote the flag $2$-sphere by $T$ and the quadrilateral enclosed by $a\mbox{-}c\mbox{-}b\mbox{-}d\mbox{-}a$ by $Q$, although both objects are altered in the process.
	
  \begin{enumerate}

    \item Denote all neighbours of $a$ in $Q$ from left to right by $c=a_0, a_1, \ldots , a_m=d$.

    \item If $a_0$ and $a_m$ have a common neighbour in $Q$ other than $a$ and $b$ we are done.
		
    \item If no such neighbour exists, let $1 \leq j \leq m-1$ be the largest index for which $a_0$ and $a_{j}$ have common neighbours outside the star of $a$.

      By the planarity of $Q$, there exist an outermost neighbour $x_1$ in $Q$, bounding a quadrilateral $x_1\mbox{-}a_0\mbox{-}a\mbox{-}a_j\mbox{-}x_1$ that contains all other common neighbours of $a_0$ and $a_j$. Note that, in this case, $a_j$ must be of degree at least five. If $x_1 = b$, $a$ and $b$ have a common neighbour and we are done. If $x_1 \neq b$, then there is at least one triangle inside $Q$ containing $a_0$ but not contained in the quadrilateral inside $Q$ and bounded by $x_1\mbox{-}a_0\mbox{-}a\mbox{-}a_j\mbox{-}x_1$. In particular, $a_0$ is of degree at least five in $T$ (however $T$ might have changed during this proof).
		
    \item If the quadrilateral inside $Q$ and bounded by $x_1\mbox{-}a_0\mbox{-}a\mbox{-}a_j\mbox{-}x_1$ does not contain interior vertices, we must have $j=1$ and the quadrilateral consists of the two triangles $a_0a_1a$ and $a_0a_1x_1$. Note that $x_1 \neq a_i$ by the flagness of the triangulation, and $x_1 a_i$ is a non-edge for $2 \leq i \leq m$ by construction of the procedure.
		
      As explained in detail above, both $a_0$ and $a_1$ are of degree at least five, and $a$ and $x_1$ do not have common neighbours other than $a_0$ and $a_1$. Hence, we can perform flip $a_0 a_1 \mapsto a x_1$ which strictly increases the degree of $a$ inside $Q$. We then start over at step 1 with $a'_0 = a_0, a'_1 = x_1, a'_2 = a_1 , \ldots a'_{m+1} = a_m$.

    \item If the quadrilateral inside $Q$ and bounded by $x_1\mbox{-}a_0\mbox{-}a\mbox{-}a_j\mbox{-}x_1$, say $Q_1$, contains interior vertices, we have $\deg_T (x_1) \geq 5$. Moreover, as explained above $\deg_T (a_0), \deg_T (a_j) \geq 5$, and $\deg_T (a) \geq 5$ by assumption. In particular, $Q_1$ is a proper quadrilateral with fewer interior vertices than $Q$. We can thus use the induction hypothesis to rearrange the interior of $Q_1$ to contain only interior vertices of degree four. Note that, in the new triangulation, all of $x_1$, $a_0$, $a$ and $a_j$ still have degree at least five (i.e., the rearranged quadrilateral is an ordered proper quadrilateral). This is important later on in the proof.
		
    \item After rearranging $Q_1$, bounded by $x_1\mbox{-}a_0\mbox{-}a\mbox{-}a_j\mbox{-}x_1$, into an ordered proper quadrilateral, repeat steps 3 to 5 by looking for the largest index $j < \ell \leq m-1$ for which $a_j$ and $a_{\ell}$ have common neighbours outside the star of $a$. Note that, whenever we flip an edge in step 4 we start over at step 1 with a strictly larger degree of vertex $a$ in $Q$.	
  \end{enumerate}
	
  This process either yields the desired result, or it terminates with $Q$ having a sequence of smaller ordered quadrilaterals $Q_1, \ldots , Q_p$ around vertex $a$, $p > 1$, see Figure~\ref{fig:quads}.
	
  Call the ``peaks'' of the quadrilaterals $x_1 , \ldots , x_p$, and the ``valleys'' between quadrilaterals $a_0 = y_0 , \ldots , y_{p} = a_m$ (cf. Figure~\ref{fig:quads}). By construction, all $x_i$, $ 1 \leq i \leq p$, and $y_j$, $0 \leq j \leq p$ are of degree at least five (see step 5 above). That is, the quadrilaterals $Q_i$, $1 \leq i \leq p$, are ordered and proper.
	
  \begin{figure}[htbp]
    \centerline{\includegraphics[width=.7\textwidth]{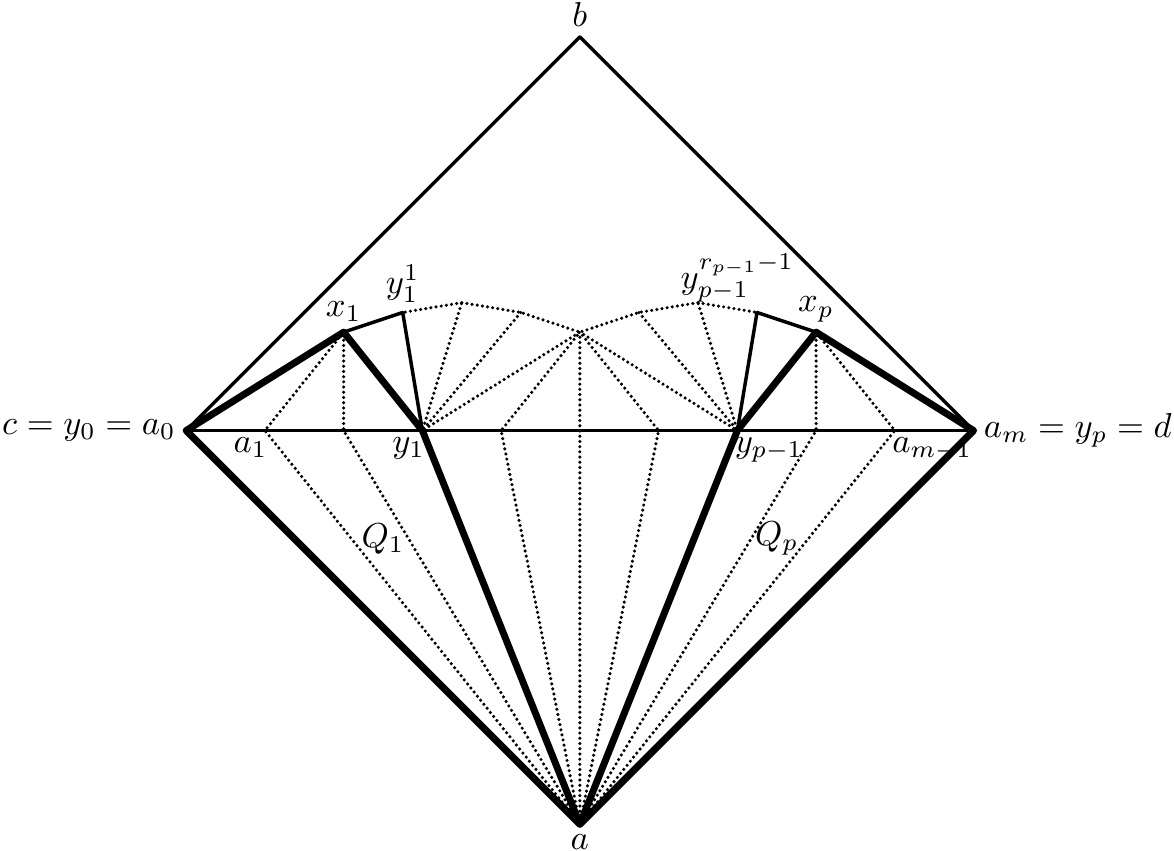}}
    \caption{The quadrilateral $Q$ after performing steps $1\mbox{-}6$, and after reorganising the interior vertices of quadrilaterals $Q_i$, $1\leq i \leq p$.\label{fig:quads}}
  \end{figure}		
	
  Recall that all quadrilaterals $Q_i$, $1\leq i \leq p$, contain only degree four interior vertices. We want all of the diagonal paths of $Q_i$, $1\leq i \leq p$, to run from $y_{i-1}$ to $y_{i}$. If $Q_i$ only has one interior vertex, this is automatically the case. Thus, assume that there exist a pair of quadrilaterals $Q_i$ and $Q_{i+1}$, $1 \leq i \leq p-1$, sharing common edge $a y_i$, and, without loss of generality, assume that $Q_i$ has a diagonal path from $a$ to $x_i$ of length at least two.

  Observe that in this particular situation, both $a$ and $y_i$ must be of degree at least six. Hence we can apply Lemma~\ref{lem:transport} to ``transport'' all but one interior vertices of $Q_i$ to the diagonal path of $Q_{i+1}$, and declare the diagonal path in $Q_i$ to run from $y_{i-1}$ to $y_i$. If the diagonal path of $Q_{i+1}$ connects $y_i$ with $y_{i+1}$ we are done. If not, note that, again, both $a$ and $y_i$ must be of degree at least six. We proceed by transporting all but one interior vertices of $Q_{i+1} $ onto the new diagonal path from $y_{i-1}$ to $y_{i}$ of $Q_{i}$, and declare the diagonal path in $Q_{i+1}$ to run from $y_{i}$ to $y_{i+1}$. Repeating this with all pairs of quadrilaterals containing at least one diagonal intersecting $a$ yields the desired result.
  Note that this procedure terminates with the degree of $a$ being at least as large as it was before starting the process at step 1 (that is, the degree of $a$ in $Q$ is at least $m+1$).
	
  \medskip	
  In Figure~\ref{fig:quads}, denote the vertices in the upper link of $y_j$ by $x_j = y_j^0, y_j^1, y_j^2, \ldots y_j^{r_j} = x_{j+1}$. By construction we have $r_j > 0$ for all $j$.
	
  Refer to Figure~\ref{fig:final}(a). Since $p>1$, $x_1$, $y_1 = a_j$, and $x_2$ are in the interior of $Q$. Moreover, both $a_j = y_1$, $j > 1$, and $x_{2}$ are of degree at least five and, by design of the procedure, $y_{0} y_1^{\ell}$, $1 \leq \ell \leq r_1$, is a non-edge (otherwise $y_1^{\ell}$ is a better choice for $x_1$). It follows that we can perform the flips $x_1 a_j \mapsto a_{j-1} y_1^1$, $x_1 a_{j-1} \mapsto a_{j-2} y_1^1$, etc., all the way down to $x_1 a_2 \mapsto a_1 y_1^1$ (see Figure~\ref{fig:final}(b)). Note that $y_1$ and $x_1$ are now both of degree at least four, the degree of $y_1^1$ is larger than before, $a_1$ is of degree five, and all other degrees have not changed. Since $x_1 a_i$, $2 \leq i \leq m$, must be non-edges, $a$ and $x_1$ do not have common neighbours. We can thus perform the flip $a_0 a_1 \mapsto a x_1$, see Figure~\ref{fig:final}(c).
	
  \begin{figure}[htbp]
    \centerline{\includegraphics[width=\textwidth]{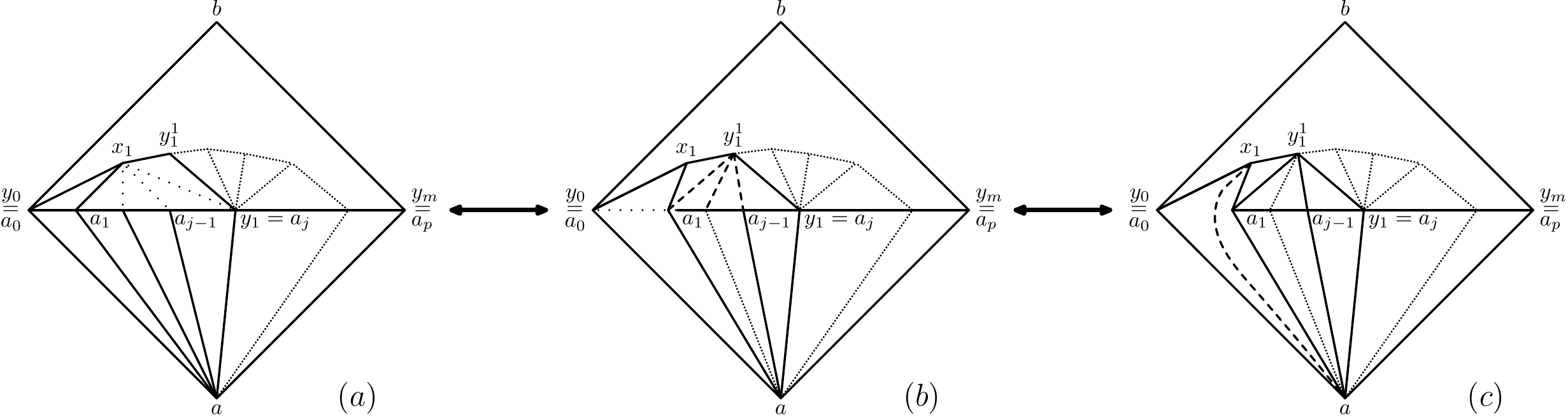}}
    \caption{Increasing the size of the link of $a$.\label{fig:final}}
  \end{figure}		
	
  This strictly increases the degree of $a$. We now start over with our procedure at step 1.

  Since there are only finitely many vertices inside $Q$, this procedure must terminate with $Q$ containing a common neighbour of $a$ and $b$. This proves the claim and completes the proof of the lemma.
\end{proof}

\begin{proof}[Proof of Theorem~\ref{thm:pachnerGraph}]
  To prove the theorem it suffices to show that $T \sim A_n$ for all $T \in \mathcal{F}_n \setminus \{ \Gamma_n \}$.

  Apply Lemma~\ref{lem:induced} to split $T$ into two proper quadrilaterals $T = Q \cup R$. This is always possible since $T \neq \Gamma_n$. Use Lemma~\ref{lem:organise} to turn all interior vertices of both $Q$ and $R$ into vertices of degree four.

  If, after the first or second application of Lemma~\ref{lem:organise}, any of the boundary vertices of $Q$ (or $R$) are of degree four, we grow $Q$ (or $R$) such that eventually it is bounded by vertices of degree at least five, or $T \sim \Gamma_n$. However, since all edge flips on $\Gamma_n$ produce a non-flag $2$-sphere triangulation, the latter case implies $T = \Gamma_n$, a contradiction.

  Thus, $T$ can be transformed into a triangulation $T''$ of the $2$-sphere which splits into two ordered proper quadrilaterals. This corresponds to the cases $a=b$ in the proof of Lemma~\ref{lem:transport}. In particular, either $T'' = \Gamma_n$, which is impossible, $T'' = A_n$, or the degrees of all vertices of the separating induced $4$-cycle satisfy the preconditions of Lemma~\ref{lem:transport}, and we can conclude that $T \sim A_n$.
\end{proof}

\section{The Pachner graph \boldmath{$\mathcal{S}_n$} of \boldmath{$n$}-vertex stacked $2$-spheres} \label{sec:stacked}

Every pair of $n$-vertex stacked $2$-spheres is, by definition, connected in the Pachner graph of stacked $2$-spheres by a sequence of $(n-4)$ $2$-moves, followed by a sequence of $(n-4)$ $0$-moves. However, if we look at the Pachner graph  $\mathcal{S}_n$ of $n$-vertex stacked $2$-spheres, the situation is different.

In this section we show that the structure of $\mathcal{S}_n$ is very special. More precisely, we prove that $\mathcal{S}_n$ is not connected for $n\geq 7$ (Corollary~\ref{coro:cor3}), and that the number of connected components rapidly increases with the number of vertices (Corollary~\ref{coro:cor4}). More precisely, for $n$ fixed, the number of connected components is at least as large as the number of isomorphism classes of trees of maximum degree at most four on $\lfloor\frac{n-5}{3}\rfloor$ vertices. See Table~\ref{fig:stackedPachner} for the number and cardinalities of connected components of $\mathcal{S}_n$ for $n \leq 14$.

\begin{table}[htb]
\begin{center}
\begin{tabular}{|r|r|r|l|}
	\hline
	$n$		& \#($\mathcal{S}_n$) & \# cc & size of connected components \\
	\hline
	\hline
    $4$		& 	$1$		& 	$1$		& 	$1$ \\
	\hline
    $5$		& 	$1$		& 	$1$		& 	$1$ \\
	\hline
    $6$		& 	$1$		& 	$1$		& 	$1$ \\
	\hline
	$7$	&	$3$	&	$1$	&	$3$\\
	\hline
	$8$	&	$7$	&	$2$	&	$1,6$\\
	\hline
	$9$	&	$24$	&	$2$	&	$1,23 $\\
	\hline
	$10$	&	$93$	&	$3$	&	$3, 4, 86 $\\
	\hline
	$11$	&	$434$	&	$5$	&	$1, 7, 10, 19, 397 $\\
	\hline
	$12$	&	$2110$	&	$8$	&	$1, 2, 6, 43, 46, 57, 82, 1873$\\
	\hline
	$13$	&	$11002$	&	$15$	&	$1, 2, 2, 3, 4, 6, 6, 7, 57, 222$ \\
  &&&                       $223, 246, 326, 394, 9503$\\
	\hline
	$14$ & $58713$	& $33$ & $1, 1, 3, 4, 4, 4, 5, 6, 6, 6, 6, 7, 7, 9, 9, 9, 12,$\\
  &&&                       $ 15, 19, 27, 28, 36, 36, 246, 304,  339, 757,$ \\
  &&&                       $ 1165, 1182, 1571, 1944, 1987, 48958$\\
	\hline
\end{tabular}
\end{center}
\caption{Number and cardinalities of the connected components of $\mathcal{S}_n$ for $n \leq 14$. \label{fig:stackedPachner}}
\end{table}

For a stacked $2$-sphere $S$, let $\SB$ be the unique stacked $3$-ball whose boundary is $S$, see Lemma~\ref{prop:bd-lbt}. If $\alpha$ is a triangle of $S$ then $\alpha$ is a face of a unique tetrahedron of $\SB$ (i.e., a clique of size four in the edge graph of $S$). We denote this unique tetrahedron by $\thickbar{\alpha}$. Naturally, $\thickbar{\alpha}$ is a node in the dual graph $\Lambda(\SB)$.

\begin{theorem} \label{theo:bsm-s2s}
Let $S$ be a stacked $2$-sphere. Let $\alpha=abc$, $\beta=abd$ be two triangles of $S$. Let $\thickbar{\alpha}$ (resp., $\thickbar{\beta}$) be the unique tetrahedron in $\SB$ containing $\alpha$ (resp., $\beta$). Then $cd$ is not an edge of $S$ and the $2$-sphere $T$ obtained from $S$ by the edge flip $ab\mapsto cd$ is stacked if and only if the nodes $\thickbar{\alpha}$ and $\thickbar{\beta}$ of $\Lambda (\SB)$ are adjacent in $\Lambda(\SB)$.
\end{theorem}

\begin{proof}
Suppose $\thickbar{\alpha}$ and $\thickbar{\beta}$ are adjacent in the dual graph $\Lambda(\SB)$, $\thickbar{\alpha} \neq \thickbar{\beta}$.
Then there exists a vertex $e$ of $S$ such that $\thickbar{\alpha}=abce$ and $\thickbar{\beta}=abde$ ($e \not\in\{d, c\}$ since $\thickbar{\alpha} \neq \thickbar{\beta}$). If $cd$ is an edge of $S$ then $\{a, b, c, d, e\}$ is  a clique in the edge graph of $S$ and hence, by Lemma~\ref{prop:bd-lbt}, $abcde$ is a simplex of $\SB$. This is not possible since $\SB$ is $3$-dimensional.

Let $B = \SB \cup abcd$. Since $\SB \cap abcd$ is a $2$-disk, $B$ is a triangulated $3$-ball. The link ${\rm lk}_{B}(ab)$ is the induced $3$-cycle $c\mbox{-} d\mbox{-}e\mbox{-} c$ in $B$. Let $D$ be obtained from $B$ by the $3$-dimensional bistellar $2$-move that replaces the three tetrahedra $abcd$, $abce$ and $abde$ around edge $ab$ with the two tetrahedra $\thickbar{\gamma}=acde$ and $\thickbar{\delta}=bcde$ sharing triangle $cde$, denoted by $ab\mapsto cde$ in short. By construction we have (i) $\partial D=T$, where $T$ is the $2$-sphere obtained from $S$ by the edge flip $ab \mapsto cd$ and (ii) all edges of $D$ are boundary edges ($ab$ is the only edge of $B$ not in the boundary which is removed by the bistellar move $ab\mapsto cde$) and thus $T$ is stacked (cf. Section~\ref{ssec:stacked}).

%Let $\xi$ be a tetrahedron of $\SB$ adjacent to the node $\thickbar{\alpha}$ in $\Lambda(\SB)$. Then $\xi$ is of the form $acex$ or $bcex$ for some vertex $x$ (note that $abc\in\partial B$).
%In the first case, $\xi$ is adjacent to the node $\thickbar{\gamma}$ in $\Lambda(\TB)$ and in the second case, $\xi$ is adjacent to the node $\thickbar{\delta}$ in $\Lambda(\TB)$
%(see Figure~\ref{fig:tree} and Remark~\ref{remark:newtree} below).
%Similarly, a neighbour of $\thickbar{\beta}$ in $\Lambda(\SB)$ is adjacent to exactly one of $\thickbar{\gamma}$ and $\thickbar{\delta}$ in $\Lambda(\TB)$. Since $\thickbar{\gamma}$ is adjacent to $%\thickbar{\delta}$ in $\Lambda(\TB)$, it follows that $\Lambda(D)$ is connected and has exactly as many arcs and nodes as $\Lambda(\SB)$. Hence $\Lambda(D)$ is a tree.
%Now, $f_0(D) = f_0(\SB) = f_3(\SB)+3 = f_3(D)+3$. Therefore, by Proposition~\ref{prop:ds-stackedball} (ii), $D$ is a stacked $3$-ball and hence $T = \partial D$ is a stacked $2$-sphere.

\medskip

Conversely, suppose $cd$ is not an edge of $S$ and the triangulated $2$-sphere $T$ obtained from the stacked $2$-sphere $S$ by the edge flip $ab\mapsto cd$ is a stacked $2$-sphere. Observe that both $\gamma = acd$ and $\delta = bcd$ are triangles of $T$.

Since $ab, abc, abd\in S = \partial \SB$, ${\rm lk}_{\SB}(ab)$ is a path in $E(\SB)$ from $c$ to $d$. Let ${\rm lk}_{\SB}(ab) = e_0\mbox{-}e_1\mbox{-}\cdots \mbox{-} e_k\mbox{-} e_{k+1}$ for some $k\geq 1$, where $e_0=c$ and $e_{k+1} =d$. We have that $abce_1 = abe_0e_1$, $abe_1e_2, \dots, abe_{k-1}e_k$, $abe_ke_{k+1} = abde_k$ are tetrahedra in $\SB$. Thus, $abe_1, \dots, abe_k$ are interior triangles of $\SB$. By Lemma~\ref{coro:interiorface}, $\SB[V(S)  \setminus\{a, b, e_1\}]$ has two components, one contains $e_0$ and the other contains $e_2$. Thus, the common neighbours of $e_0$ and $e_2$ in $E(\SB) = E(S)$ are $a, b$ and $e_1$. Similarly, the set of common neighbours of $e_{i-1}$ and $e_{i+1}$ is $\{a, b, e_i\}$ for $1\leq i\leq k$. This implies that the set of common neighbours of $c=e_0$ and $d=e_{k+1}$ in $E(T)$ is $\{a, b, e_1\} \cap \{a, b, e_k\}$ (note that $E(S)$ differs from $E(T) = E(\TB)$ only in edges $ab$ and $cd$).

On the other hand the triangles $\gamma = acd = ae_0e_{k+1}$ and $\delta = bcd = be_0e_{k+1}$ are contained in unique tetrahedra $\thickbar{\gamma} = acdx$ and $\thickbar{\delta} = bcdy$ of $\TB$ and hence $a$, $b$, $x$ and $y$ are common neighbours of $c$ and $d$. By the above this is only possible if $e:=x=y=e_1 = e_k$. In particular, ${\rm lk}_{\SB}(ab)$ is a path from $c$ to $d$ of length two, $\thickbar{\alpha}=abce$, $\thickbar{\beta}=abde$, and in particular $\thickbar{\alpha}$ and $\thickbar{\beta}$ are adjacent in $\Lambda(\SB)$.
\end{proof}

\begin{remark}
  For an edge flip $ab \mapsto cd$ on a stacked $2$-sphere $S$ to be valid, we must have $\alpha=abc, \beta=abd \in F(S)$ and $cd\not\in E(S)$. We have seen that an $n$-vertex $2$-sphere $T$ can be obtained from a stacked $2$-sphere $S$ by an edge flip $ab \mapsto cd$ (that is, the edge flip is valid) and $T$ is stacked if and only if the nodes corresponding to tetrahedra $\thickbar{\alpha}$ and $\thickbar{\beta}$ of $\SB$ are adjacent in $\Lambda (\SB)$.

  Note that we can replace this latter condition in Theorem~\ref{theo:bsm-s2s} by any of the following equivalent conditions (some are immediate, some follow from Lemma~\ref{coro:interiorface}):
  \begin{itemize}
    \item The path in the link of $ab$ from $c$ to $d$ is of length exactly two.
    \item Edge $ab$ is contained in exactly two tetrahedra of $\SB$.
    \item The vertices $a$ and $b$ have exactly three common neighbours in $S$.
    \item There exists a unique vertex $e \not\in\{c, d\}$ such that $ae$ and $be$ are edges of $S$.
  \end{itemize}
  While some of these conditions are easier to grasp, others are more efficient for implementations. It is thus useful to keep all of them in mind.
\end{remark}

\begin{remark} \label{remark:newtree}
Let $T$ be obtained from $S$ by the edge flip $ab\mapsto cd$ and $e$, $\alpha = abc$, $\beta = abd$, $\gamma = acd$, $\delta = bcd$ as in the proof of Theorem~\ref{theo:bsm-s2s}. Then $\thickbar{\alpha} = abce$, $\thickbar{\beta}= abde\in \SB$ and $\thickbar{\gamma} = acde$, $\thickbar{\delta}= bcde \in \TB$. Moreover, let the (up to) two nodes adjacent to $\thickbar{\alpha}$ in $\Lambda (\SB)$ be $acex$ and $bcey$, and let the (up to) two nodes adjacent to $\thickbar{\beta}$ in $\Lambda (\SB)$ be $adez$ and $bdew$.

\begin{figure}[htbp]
  \begin{center}
    \includegraphics[width=0.5\textwidth]{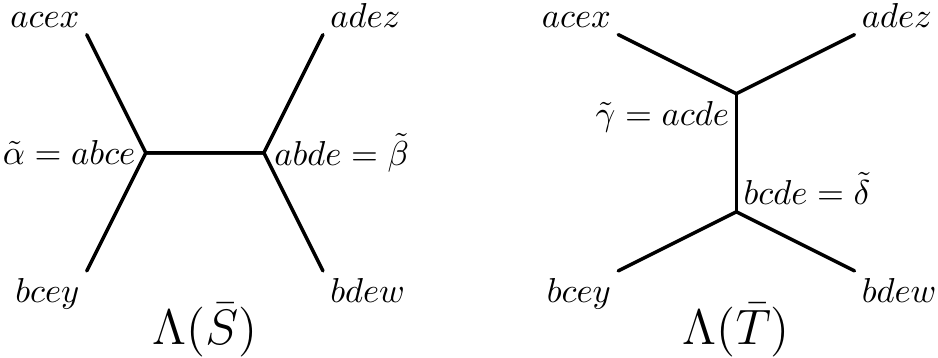}
    \caption{Transformation of dual graph by edge flip $ab \mapsto cd$ in the proof of Theorem~\ref{theo:bsm-s2s}. \label{fig:tree}}
  \end{center}
\end{figure}

Then the dual graph $\Lambda(\TB)$ is the tree build from $\Lambda (\SB)$, with set of nodes $U = \{\sigma \in \SB \,\mid\, \sigma \textrm{ is a tetrahedron }\} \setminus\{\thickbar{\alpha}, \thickbar{\beta}\}) \cup \{\thickbar{\gamma}, \thickbar{\delta}\}$ with all arcs in $\Lambda (\SB)$ adjacent to $\thickbar{\alpha}$ and $\thickbar{\beta}$ removed, and arcs added between $\thickbar{\gamma}$ and $\thickbar{\delta}$ (corresponding to triangle $cde$), $\thickbar{\gamma}$ and $acex$ (corresponding to $ace$), $\thickbar{\delta}$ and $bcey$ ($bce$), $\thickbar{\gamma}$ and $adez$ ($ade$), and  $\thickbar{\delta}$ and $bdew$ ($bde$), see Figure~\ref{fig:tree}.
\end{remark}

\begin{kor} \label{coro:cor1}
Let $S$ be a stacked $2$-sphere, $\alpha=abc$, $\beta=abd$ two triangles of $S$, $\thickbar{\alpha}$ (resp., $\thickbar{\beta}$) the unique tetrahedron of $\SB$ containing $\alpha$ (resp., $\beta$), $\sigma \in \SB$ correspond to a degree four node in $\Lambda(\SB)$, and let $G_1, G_2, G_3, G_4$ be the connected components of $\Lambda(\SB)-\sigma$. If the $2$-sphere $T$ obtained from $S$ by the edge flip $ab\mapsto cd$ is also a stacked $2$-sphere then
\begin{enumerate}[{\rm (i)}]
  \item $\sigma$ is a tetrahedron of $\TB$,
  \item $\sigma$ is a degree four node in $\Lambda(\TB)$,
  \item both $\thickbar{\alpha}$ and $\thickbar{\beta}$ are in one component of $\Lambda(\SB)-\sigma$, say in $G_4$, and
  \item the components of $\Lambda(\TB)-\sigma$ are $G_1, G_2, G_3, G_4^{\prime}$ for some tree $G_4^{\prime}$.
\end{enumerate}
\end{kor}

\begin{proof}
It follows from Theorem~\ref{theo:bsm-s2s} that ${\rm lk}_{\SB}(ab)$ is a path of the from $c\mbox{-}e\mbox{-}d$ and $\thickbar{\alpha} = abce$, $\thickbar{\beta}= abde$ for some vertex $e$. In particular, $\thickbar{\alpha}$ and $\thickbar{\beta}$ are the only two tetrahedra in $\SB$ containing $ab$. Since all the $2$-dimensional faces of $\sigma$ are interior triangles, we have $\sigma\not \in \{\thickbar{\alpha}, \thickbar{\beta} \}$. Thus, $\sigma$ cannot contain the edge $ab$. Since $\sigma$ forms a clique in $E(S)$, this implies that $\sigma$ forms a clique in $E(T)$. Hence $\sigma\in\TB$. This proves part (i).

Observe that $\{a, c, d, e\}$ and $\{b, c, d, e\}$ span cliques in $E(T)$.
Therefore, $\thickbar{\gamma}:=acde, \thickbar{\delta}:=bcde\in\TB$.
Let $\tau$ be a $2$-dimensional face of $\sigma$. Then $\tau$ is an interior face in $\SB$. Let $\tau= \sigma\cap\mu$ for some tetrahedron $\mu\in\SB$. If $ab\not\subset\mu$ then $\mu$ forms a clique in $E(T)$ and hence $\mu\in\TB$. Then $\tau= \sigma\cap\mu$ is an interior triangle of $\TB$. If $ab\subset\mu$ then $\mu$ is
$\thickbar{\alpha}$ or $\thickbar{\beta}$. Assume, without loss, that $\mu= \thickbar{\alpha}=abce$. Since $ab\not\subset\sigma$ and $\mu\cap\sigma$ is a  face of $\mu$, $\tau=\mu\cap\sigma=ace$ or $bce$. Assume, without loss, that $\tau=ace$. Then $\sigma = acex$ for some vertex $x$ and $\tau = \sigma\cap\thickbar{\gamma}$. Thus, $\tau$ is an interior triangle of $\TB$. Thus, each $2$-dimensional face of $\sigma$ is an interior triangle of $\TB$. Part (ii) follows from this.

Part (iii) follows from the fact that $\thickbar{\alpha}$ and $\thickbar{\beta}$ share a triangle in $\SB$ which (necessarily) is not a face of $\sigma$.

The four $2$-dimensional faces of $\thickbar{\gamma} =acde$ are $acd, ace, ade$ and $cde$. Since $cd$ is a non-edge in $\SB$, we have that $acd, cde$ are not in $\SB$ and $ace = \thickbar{\gamma}\cap\thickbar{\alpha}$, $ade = \thickbar{\gamma}\cap \thickbar{\beta}$. Thus, by part (iii), $\thickbar{\gamma}$ is not adjacent to any nodes of $G_1 \cup G_2\cup G_3$. Similarly, $\thickbar{\delta}$ is not adjacent to any nodes of $G_1 \cup G_2\cup G_3$. Part (iv) now follows since the set of nodes of $\Lambda(\TB)$ is $(\{\tau :$ $\tau$ is a tetrahedron in $\SB\}\setminus\{\thickbar{\alpha}, \thickbar{\beta}\})\cup\{\thickbar{\gamma},  \thickbar{\delta}\}$.
\end{proof}

\begin{kor} \label{coro:cor2}
Let $S$ be a stacked $2$-sphere, $T$ a stacked $2$-sphere obtained from $S$ by an edge flip, and let $V_S$ (resp., $V_T$) be the set of degree four nodes in $\Lambda(\SB)$ (resp., in $\Lambda(\TB)$). Then the induced subgraphs $\Lambda(\SB)[V_S]$ and $\Lambda(\TB)[V_T]$ are isomorphic.
\end{kor}

\begin{proof}
By Corollary~\ref{coro:cor1}, $V_S=V_T$. For $\sigma_1, \sigma_2\in V_S=V_T$, $\sigma_1$ and $\sigma_2$ are adjacent in $\Lambda(\SB)[V_S]$ if and only if $\sigma_1\cap\sigma_2$ is an interior triangle of $\SB$ if and only if $\sigma_1\cap\sigma_2$ contains three vertices if and only if $\sigma_1\cap\sigma_2$ is an interior triangle of $\TB$ if and only if $\sigma_1$ and $\sigma_2$ are adjacent in $\Lambda(\TB)[V_T]$. The corollary follows from this observation.
\end{proof}

\begin{kor} \label{coro:cor3}
  The Pachner graph $\mathcal{S}_n$ of $n$-vertex stacked $2$-spheres  is disconnected for $n \geq 8$.
\end{kor}

\begin{proof}
The stacked $3$-ball associated to an $n$-vertex stacked $2$-sphere, $n\geq 8$, has a dual graph with $m = n-3 \geq 5$ nodes, and every $m$-node tree (with degrees of nodes $\leq 4$) is the dual graph of at least one stacked $3$-ball. Hence there exist a stacked $3$-ball $B_1$ with dual graph having one node of degree four and $m-1$ nodes of degree at most three, and there exist a stacked $3$-ball $B_2$ with dual graph with all $m$ nodes of degree at most two. Then, by Corollary~\ref{coro:cor2}, the $n$-vertex stacked $2$-spheres $\partial B_1$ and $\partial B_2$ are in different connected components of $\mathcal{S}_n$.
\end{proof}

\begin{kor} \label{coro:cor4}
For $m\in \mathbb{Z}^{+}$, let $t(m)$ be the number of non-isomorphic $m$-node trees with degrees of nodes at most four. Moreover, let $n=3m+5$. Then the Pachner graph  $\mathcal{S}_n$ of $n$-vertex stacked $2$-spheres has $t(m)$ components each containing a single stacked $2$-sphere.
\end{kor}

\begin{proof}
Let $H$ be an $m$-node tree in which degrees of all the nodes are at most four. Consider a new graph $G$ by connecting each node of $H$ of degree $i$ to $(4-i)$ new nodes. Then $G$ is a connected acyclic graph and hence a tree. By construction, the number of new nodes in $G$ equals the number of new arcs in $G$ which is $\sum_{v\in V(H)}(4-\deg_H(v)) = 4m - \sum_{v\in V(H)}\deg_H(v)= 4m - 2(m-1)= 2m+2$. Therefore, $G$ has $(m-1)+(2m+2) = 3m+1$ arcs, and thus $3m+2$ nodes. It follows that $G$ has $m$ nodes of degree four and $2m+2$ nodes of degree one, and each degree one node of $G$ is adjacent to a degree four node.

Let $B$ be a stacked $3$-ball whose dual graph $\Lambda(B)$ is $G$. It follows from the definition that we can always construct such a stacked $3$-ball. Let $S= \partial B$. Since $S$ is stacked it must have $3m +5$ vertices. Let $\alpha=abc$, $\beta=abd$ be two triangles of $S$, and let $\thickbar{\alpha}$ (resp., $\thickbar{\beta}$) be the unique tetrahedron of $B$ containing $\alpha$ (resp., $\beta$). Then $\deg_{\Lambda(B)}(\thickbar{\alpha}),  \deg_{\Lambda(B)}(\thickbar{\beta})< 4$ and hence $\deg_{\Lambda(B)}(\thickbar{\alpha}) = 1 = \deg_{\Lambda(B)}(\thickbar{\beta})$. If $\thickbar{\alpha} =\thickbar{\beta}$, then $cd$ is an edge and hence we cannot perform the edge flip $ab\mapsto cd$.
If $\thickbar{\alpha} \neq\thickbar{\beta}$, then $\thickbar{\alpha}$ and $\thickbar{\beta}$ are not adjacent in $\Lambda(B)$ (degree one nodes are only adjacent to degree four nodes in $\Lambda(B)$) and hence, by Theorem~\ref{theo:bsm-s2s}, the $2$-sphere $T$ obtained from $S$ by the edge flip $ab\mapsto cd$ is not stacked. Thus $S$ is isolated in $\mathcal{S}_n$.

If $H_1$ and $H_2$ are non-isomorphic trees on $m$ nodes, then the above construction carried out for both $H_1$ and $H_2$ leads to two non-isomorphic trees $G_1$ and $G_2$, leading to two non-isomorphic stacked $3$-balls $B_1$ and $B_2$ with, by Lemma~\ref{prop:bd-lbt}, non-isomorphic boundaries $S_1$ and $S_2$. Since there exist at least $t(m)$ non-isomorphic $m$-node trees with degree of nodes at most four, we have at least $t(m)$ singleton components in $\mathcal{S}_n$.
\end{proof}

\begin{kor} \label{coro:cor5}
The number of connected components in $\mathcal{S}_n$ is bounded from below by $C^n$, for some real number $C>1$.
\end{kor}

\begin{proof}
Let $m=\lfloor\frac{n-5}{3}\rfloor$. Let $t(m)$ be the number of non-isomorphic $m$-node trees with degree of nodes at most four as in Corollary~\ref{coro:cor4}.

\medskip

\noindent {\bf Claim}: The number of components in $\mathcal{S}_n$ is at least $t(m)$.

\smallskip

Let $\mathcal{T}$ be the set of all $m$-node trees with node-degrees at most four.
For each $H\in \mathcal{T}$, we can construct a $(3m+5)$-nodes tree $G$ whose degree four nodes are the nodes of $H$ and all others are of degree 1. By adding $n-3m-5$ ($\leq 2$) new nodes to the $n-3m-2$ degree one nodes of $G$ we obtain a new tree $G^{\prime}$ having the same set of degree four nodes as in $G$. Let $B$ be a stacked $3$-ball whose dual graph is $G^{\prime}$ and let $S=\partial B$. By construction, $S$ is a stacked $2$-sphere with exactly $n$ vertices. Let $V_S$ be as in Corollary~\ref{coro:cor2}. Then $G^{\prime}[V_S]=G[V_S]=H$. Therefore, by Corollary~\ref{coro:cor2}, the $n$-vertex stacked $2$-spheres obtained in this process corresponding to different graphs in $\mathcal{T}$ are in different components of $\mathcal{S}_n$. This proves the claim.

\smallskip

Since $t(m)$ is exponential in $m$, the result follows from the claim.
\end{proof}

Following arguments along the lines of Corollary~\ref{coro:cor1} we can observe that, apart from a large number of isolated singleton components in $\mathcal{S}_n$, there are also larger connected components corresponding to dual graphs with no, or very few nodes of degree four. For instance, the largest connected component in $\mathcal{S}_n$, $n \leq 14$, shown in Table~\ref{fig:stackedPachner}, corresponds to boundaries $S$ of stacked balls $\SB$ with dual graphs without nodes of degree four (i.e., $V_S = \emptyset$). Let $\mathcal{S}_n^0$ denote the Pachner graph consisting of this class of stacked $2$-spheres. We have the following result.

\begin{theorem}
  \label{thm:stackedcomp}
  The Pachner graph $\mathcal{S}_n^0$ is connected.
\end{theorem}

We split the proof of Theorem~\ref{thm:stackedcomp} into two lemmas.

\begin{lemma}
  \label{lemma:49a}
  Each stacked $2$-sphere $S \in \mathcal{S}_n^0$ is connected to a stacked $2$-sphere $T$ in the Pachner graph $\mathcal{S}_n^0$, where the dual graph $\Lambda(\TB)$ of $\TB$ is a path.
\end{lemma}

\begin{proof}
The idea of the proof is to show that, for every $S\in\mathcal{S}_n^0$ with $\Lambda(\SB)$ not a path, $S$ is connected in $\mathcal{S}_n^0$ to a stacked $2$-sphere $T\in\mathcal{S}_n^0$ with the number of nodes of degree three in $\Lambda(\TB)$ less than that in $\Lambda(\SB)$.

For $S\in \mathcal{S}_n^0$ and $\alpha$, $\beta$ nodes in $\Lambda(\SB)$, let $d_S(\alpha, \beta)$ be the length of the unique path from $\alpha$ to $\beta$ in the tree $\Lambda(\SB)$. Moreover, if $S$ has a degree three node in $\Lambda(\SB)$, let $\ell(S) = \min\{d_S(\alpha, \beta) \, \mid \, \alpha$ leaf, $\beta$ degree three in $\Lambda(\SB)\}$.

\medskip

\noindent {\bf Claim 1}: Let $S\in \mathcal{S}_n^0$ be a stacked $2$-sphere such that $\Lambda(\SB)$ is not a path. If $\ell(S)\geq 2$ then there exists a stacked $2$-sphere $T\in \mathcal{S}_n^0$ such that (i) $S$ is connected to $T$ in $\mathcal{S}_n^0$, (ii) the number of degree three nodes in $\Lambda(\TB)$ is the same as in $\Lambda(\SB)$ and (iii) $\ell(T) = \ell(S)-1$.

\smallskip

Let $\ell= \ell(S)= d_S(\gamma,\delta)$, where $\gamma$ is a degree three node and $\delta$ is a leaf in $\Lambda(\SB)$. Let $\gamma_0\mbox{-}\gamma_1\mbox{-} \cdots \mbox{-}\gamma_{\ell}$ be the path in $\Lambda(\SB)$ from $\gamma =\gamma_0$ to $\delta=\gamma_{\ell}$. Then $\deg_{\Lambda(\SB)}(\gamma_0)=3$, $\deg_{\Lambda(\SB)}(\gamma_i)=2$ for $1\leq i\leq \ell-1$, and $\deg_{\Lambda(\SB)}(\gamma_{\ell})=1$. Let the other nodes adjacent to $\gamma$ be $\alpha$ and $\beta$. Assume, without loss, that $\gamma=1234$, $\alpha=124a$, $\beta=134b$ and $\gamma_1=123x_1$. Then, the link of $23$ in $\SB$ is of the form $4\mbox{-}1\mbox{-}x_1\mbox{-}\cdots\mbox{-}x_k$ for some $k\leq\ell$.

\smallskip

\noindent  {\bf Case 1}. Let $k=1$. It follows that $23x_1$ is a triangle of $S = \partial \SB$. By Theorem~\ref{theo:bsm-s2s}, the triangulated $2$-sphere $T$ obtained from $S$ by the edge flip $23\mapsto 4x_1$ is stacked and hence, by Corollary~\ref{coro:cor1}, is in $\mathcal{S}_n^0$. By Lemma~\ref{prop:bd-lbt}, $\gamma^{\prime} :=134x_1$ and $\gamma_1^{\prime} := 124x_1$ are tetrahedra in $\TB$.

Following the transformation of the dual graph of a stacked ball under an edge flip, as shown in Figure~\ref{fig:tree}, the dual graph $\Lambda(\TB)$ is obtained from $\Lambda(\SB)$ by replacing the three edges adjacent to $\gamma$ with the path $\beta\mbox{-}\gamma^{\prime}\mbox{-}\gamma_1^{\prime}\mbox{-}\alpha$, and attaching the path $\gamma_2\mbox{-}\cdots\mbox{-}\gamma_{\ell}$ to either $\gamma^{\prime}$ or $\gamma_1^{\prime}$. In either case, the path from the new degree three node to $\gamma_{\ell}$ is of length $\ell-1$, and since the remaining part of $\Lambda(\TB)$ is equal to the remaining part of $\Lambda(\SB)$, we have $\ell(T) = \ell(S)-1$ and Claim 1 is true in this case.

\smallskip

\noindent  {\bf Case 2}. Let $k\geq 2$. In this case we can assume that $\gamma_i= 23x_{i-1}x_i$ for $2\leq i\leq k$, and that the triangles $21x_1$, $2x_1x_2, \dots, 2x_{k-2}x_{k-1}$, $31x_1$, $3x_1x_2, \dots, 3x_{k-2}x_{k-1}\in S$ (i.e., are in the boundary of $\SB$). Since $\deg_{\Lambda(\SB)}(\gamma_k)\leq 2$ ($=1$ if $k=\ell$ and $=2$ if $k<\ell$), at least two $2$-dimensional faces of $\gamma_k$ are triangles of $S$. This implies that at least one of the triangles $2x_{k-1}x_{k}$ and $3x_{k-1}x_{k}$ is a triangle of $S$.

\begin{figure}[htbp]
  \begin{center}
    \includegraphics[width=.9\textwidth]{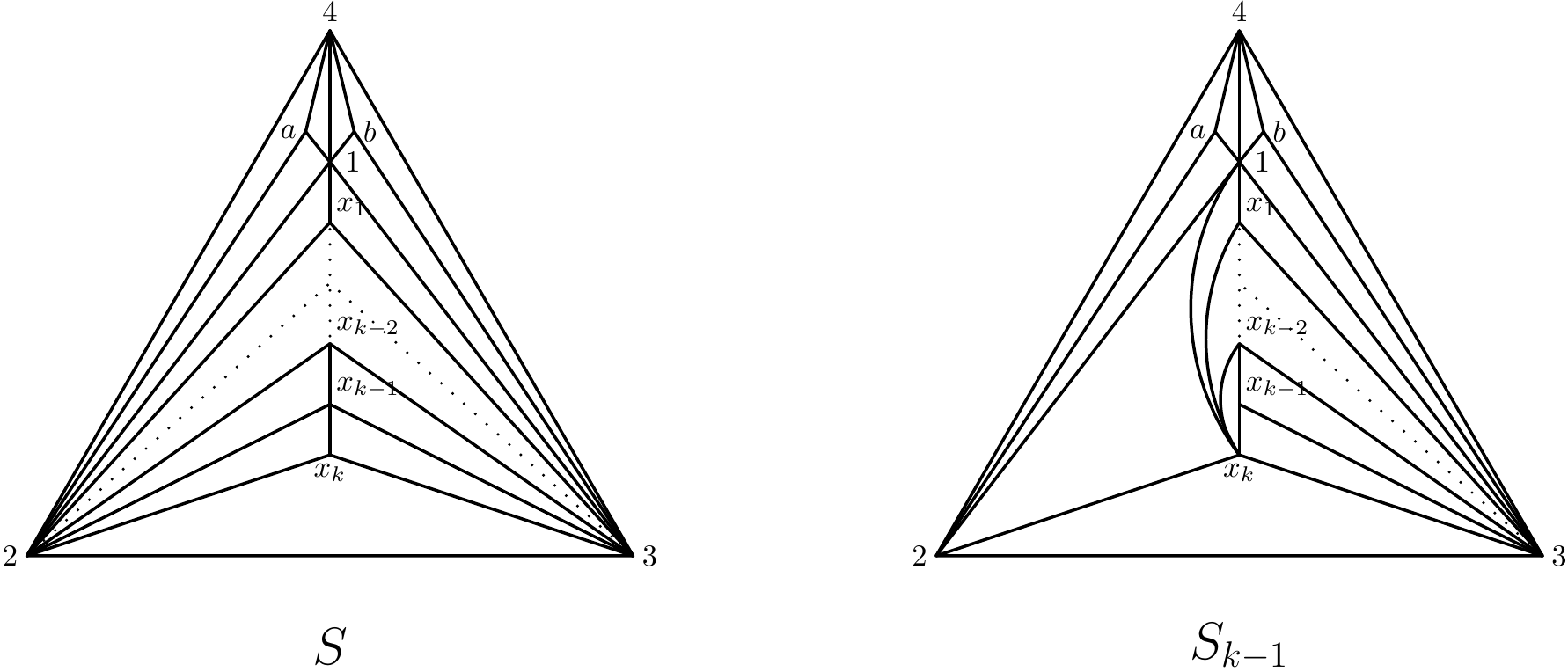}
    \caption{Sequence of edge flips as performed in the proof of Lemma~\ref{lemma:49a}, Claim 1, Case 2. \label{fig:Lem410}}
  \end{center}
\end{figure}

Assume, without loss, that $2x_{k-1}x_{k}\in S$. (In that case, $\gamma_{k+1}$ is of the form $3x_{k-1}x_{k}x_{k+1}$ for some $x_{k+1}\in V(S)$ when $k<\ell$.) 
Let $S_1$ be obtained from $S=S_0$ by the edge flip $2x_{k-1}\mapsto x_{k}x_{k-2}$. Since $\link_{\SB}(2x_{k-1})=x_{k-2}\mbox{-}3\mbox{-}x_{k}$, by Theorem~\ref{theo:bsm-s2s}, $S_1$ is stacked. Observe that the path $\gamma_{k-2}\mbox{-} \gamma_{k-1}\mbox{-}\gamma_{k} \mbox{-} \gamma_{k+1}$ in $\Lambda(\SB)$ is replaced by $\gamma_{k-2}\mbox{-} (23x_kx_{k-2}) \mbox{-} (3x_kx_{k-1}x_{k-2}) \mbox{-}\gamma_{k+1}$ in $\Lambda(\SB_1)$ when $k < \ell$,
and $\gamma_{k-2}\mbox{-} \gamma_{k-1}\mbox{-}\gamma_{k}$ is replaced by $\gamma_{k-2}\mbox{-} (23x_kx_{k-2}) \mbox{-} (3x_kx_{k-1}x_{k-2})$ when $k = \ell$.
Thus, $\Lambda(\SB_1)$ is isomorphic to $\Lambda(\SB)$.

Inductively, for $1\leq i\leq k-1$, $\link_{\SB}(2x_{k-i})= x_{k-i-1}\mbox{-} 3 \mbox{-}x_{k}$ and hence the sphere $S_{i}$ obtained from $S_{i-1}$ by the edge flip $2x_{k-i}\mapsto x_{k}x_{k-i-1}$ is stacked. Then $\Lambda(\SB_i)$ is isomorphic to $\Lambda(\SB_{i-1})$, see Figure~\ref{fig:Lem410}. (Note that $S_{k-1}$ is obtained by the sequence of edge flips $2x_{k-1}\mapsto x_{k}x_{k-2}$, $2x_{k-2}\mapsto x_{k}x_{k-3}, \dots, 2x_{2}\mapsto x_{k}x_{1}$, $2x_{1}\mapsto x_{k}1$.)

It follows that $S_{k-1}$ is stacked, $S$ can be joined to $S_{k-1}$ in $\mathcal{S}_n^0$, $\Lambda(\SB_{k-1})$ is isomorphic to $\Lambda(\SB)$, and $\link_{\SB_{k- 1}}(23) =4\mbox{-}1\mbox{-}x_{k}$. In particular, $S_{k-1}$ satisfies the hypothesis of Case 1, $\ell(S_{k-1}) = \ell(S)$ and the number of degree three nodes in $\Lambda(\SB_{k-1})$ is the same as that in $\Lambda(\SB)$. Consequently, by Case 1, $S_{k-1}$ is connected to some $T$ in $\mathcal{S}_n^0$, such that the number of degree three nodes in $\Lambda(\TB)$ is the same as that in $\Lambda(\SB_{k-1})$ (which is the same as that in $\Lambda(\SB)$) and $\ell(T) = \ell(S_{k-1})-1=\ell(S)-1$. This completes the proof of Claim 1.

\medskip

\noindent {\bf Claim 2}: For $S\in \mathcal{S}_n^0$, if $\Lambda(\SB)$ has a leaf which is adjacent to a degree three node in $\Lambda(\SB)$ (i.e., $\ell(S)=1$) then there exists $T \in \mathcal{S}_n^0$ which can be obtained from $S$ by an edge flip and the number of nodes of degree three in $\Lambda(\TB)$ is one less than that in $\Lambda(\SB)$.

\smallskip

Let $\delta=123d$ be a leaf node which is adjacent to a degree three node $\gamma= 1234$. Assume, as above, that the adjacent nodes of $\gamma$ are $\alpha=124a$ and $\beta=134b$. Then edge $23$ is in two tetrahedra and, by Theorem~\ref{theo:bsm-s2s}, the $2$-sphere $T$ obtained from $S$ by the edge flip $23\mapsto 4d$ is stacked and hence in $\mathcal{S}_n^0$ by Corollary~\ref{coro:cor1}. Moreover, by Lemma \ref{prop:bd-lbt}, $\gamma^{\prime} :=124d$ and $\delta^{\prime}:= 134d$ are in $\TB$. Again, by following the transformation shown in Figure~\ref{fig:tree}, $\Lambda(\TB)$ contains the path $\alpha\mbox{-}\gamma^{\prime} \mbox{-} \delta^{\prime}\mbox{-} \beta$ instead of the three edges adjacent to $\gamma$ in $\Lambda(\SB)$. Since the remaining parts of $\Lambda(\SB)$ and $\Lambda(\TB)$ coincide, Claim 2 follows.

\smallskip
The result follows inductively using Claims 1 and 2.
\end{proof}

\begin{lemma}
  \label{lemma:49b}
  Let $\partial \Delta_n$ be as shown in Figure~\ref{fig:deltan} and let $S \in \mathcal{S}_n^0$. If $\Lambda(\SB)$ is a path then $S$ is connected to $\partial\Delta_n$ in $\mathcal{S}_n^0$.
\end{lemma}

\begin{figure}[!h]
  \begin{center}
    \includegraphics[height=5cm]{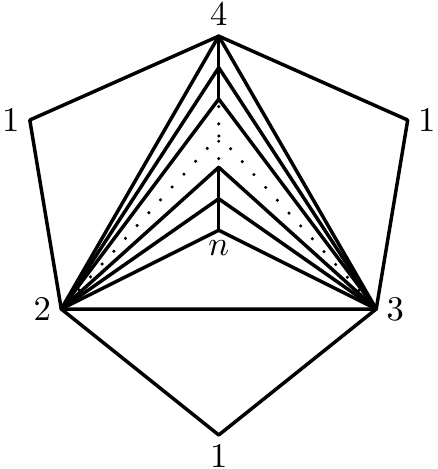}
    \caption{The canonical stacked $3$-ball $\Delta_n$. Note that this complex is also used as a canonical target in \cite{Bose11FlipGraph} to prove upper bounds on the diameter of the Pachner graph $\mathcal{P}_n$ of $n$-vertex $2$-spheres. \label{fig:deltan}}
  \end{center}
\end{figure}

\begin{proof}
Let $\Lambda(\SB) = \gamma_1\mbox{-}\gamma_2\mbox{-}\cdots \mbox{-}\gamma_{n-3}$.

\medskip

\noindent {\bf Claim}: If $\gamma_1, \dots, \gamma_k$ have a common edge and $\gamma_1, \dots, \gamma_{k+1}$ have no common edge, $k\leq n-4$, then $S$ can be joined to $T \in \mathcal{S}_n^0$, where $\Lambda(\TB)$ is a path of the form $\alpha_1\mbox{-}\alpha_2\mbox{-}\cdots \mbox{-}\alpha_{n-3}$ such that $
\alpha_1, \dots, \alpha_{k+1}$ have a common edge.

\begin{figure}[!h]
  \begin{center}
    \includegraphics[width=0.9\textwidth]{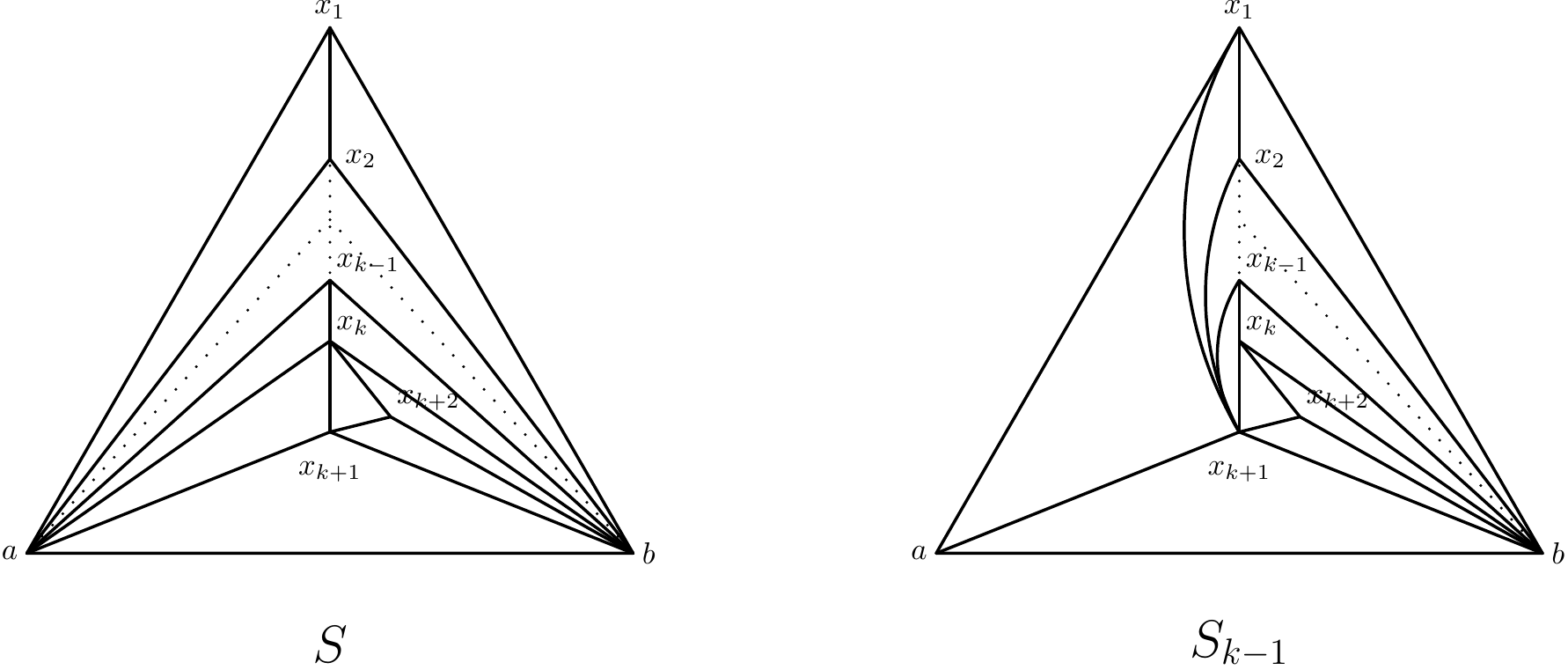}
    \caption{Sequence of edge flips as performed in the proof of Lemma~\ref{lemma:49b}. \label{fig:Lem411}}
  \end{center}
\end{figure}

Since $\gamma_1, \dots, \gamma_{k+1}$ have no common edge, we can assume that $k\geq 3$. Let $\gamma_i =abx_ix_{i+1}$ for $1\leq i\leq k$. Assume without loss of generality that $\gamma_{k+1} =bx_kx_{k+1}x_{k+2}$. Then $\link_{\SB}(ax_k)= x_{k-1}\mbox{-}b\mbox{-} x_{k+1}$. Thus, by Theorem~\ref{theo:bsm-s2s}, the $2$-sphere $S_1$ obtained from $S$ by the edge flip $ax_k\mapsto x_{k+1}x_{k-1}$ is stacked. Similarly, the $2$-sphere $S_2$ obtained from $S_1$ by the edge flip $ax_{k-1}\mapsto x_{k+1}x_{k-2}$ is stacked. Continuing this way, we obtain a stacked sphere $T=S_{k-1}$ from $S_{k-2}$ by the edge flip $ax_2\mapsto x_{k+1}x_1$, see Figure~\ref{fig:Lem411}. Hence $S$ can be joined to $T$ in $\mathcal{S}_n^0$ and $\Lambda(\TB) = \alpha_1\mbox{-}\alpha_2\mbox{-}\cdots \mbox{-}\alpha_{k+1}\mbox{-}\gamma_{k+2}\mbox{-}\cdots\mbox{-}\gamma_{n-3}$, where $\alpha_1= bx_{k+1}ax_1$, $\alpha_{i}=bx_{k+1}x_{i-1}x_{i}$, $2 \leq i \leq k$, and $\alpha_{k+1} =bx_{k+1}x_{k}x_{k+2}$. This proves the claim.

\smallskip

The lemma follows by induction using the claim.
\end{proof}

\begin{proof}[Proof of Theorem \ref{thm:stackedcomp}]
The result follows from Lemmas \ref{lemma:49a} and \ref{lemma:49b}.
\end{proof}

	{\footnotesize
	 \bibliographystyle{abbrv}
	 \bibliography{/home/jspreer/Dropbox/Bibliography/bibliography}
	}

\end{document}